\documentclass[12pt]{amsart}

\usepackage{hyperref}

\usepackage[scheme=plain]{ctex} 

\usepackage{ctex}
\usepackage{epsfig}
\usepackage{graphicx}
\usepackage{amssymb, amstext, amscd, amsmath}
\usepackage{amsthm, mathrsfs, amsfonts,dsfont}
\usepackage{fullpage}
\usepackage{txfonts}
\usepackage{fancybox}
\usepackage{color}
\usepackage{cite}
\usepackage{comment}
\newtheorem{theorem}{Theorem}[section]
\newtheorem{lemma}[theorem]{Lemma}
\newtheorem{proposition}[theorem]{Proposition}
\newtheorem{corollary}[theorem]{Corollary}
\theoremstyle{definition}
\newtheorem{definition}[theorem]{Definition}
\newtheorem{example}[theorem]{Example}

\newtheorem{question}[theorem]{Question}
\theoremstyle{remark}
\newtheorem{remark}[theorem]{Remark}
\numberwithin{equation}{section}

\newtheorem*{Thm}{\bf{Main Theorem}}

\begin{document}
\title{A geometric approach to the  compressed shift operator on the Hardy space over the bidisk}

\author{Yufeng Lu}
\address{School of Mathematics Sciences, Dalian University of Technology, Dalian, Liaoning, 116024, P. R. China}
\address{State Key Laboratory of Structural Analysis for Industrial Equipment, Dalian, Liaoning, 116024, P. R. China}
\email{lyfdlut@dlut.edu.cn}

\author{Yixin Yang}
\address{School of Mathematics Sciences, Dalian University of Technology,
Dalian, Liaoning, 116024, P. R. China}
\email{yangyixin@dlut.edu.cn}

\author{Chao Zu$^{*}$}
\address{School of Mathematics Sciences, Dalian University of Technology,
Dalian, Liaoning, 116024, P. R. China}
\email{zuchao@mail.dlut.edu.cn}



\thanks{* Corresponding author.}
\thanks{This research is supported by National Natural Science Foundation of China
(No.  12031002, 11971086). This research is also partially supported by Dalian High-level Talent innovation Project (Grant 2020RD09).}

\begin{abstract}
This paper studies the  compressed shift operator $S_z$ on  the Hardy space   over the bidisk via the geometric approach. We calculate the spectrum and essential spectrum of     $S_z$ on the Beurling type   quotient modules induced by rational inner functions, and   give a complete characterization  for $S_z^*$ to be a Cowen-Douglas operator. Then  we extend  the concept of Cowen-Douglas operator to be the generalized Cowen-Douglas operator, and show that $S_z^*$ is a generalized Cowen-Douglas operator. Moreover, we establish the connection between the reducibility of the Hermitian holomorphic vector bundle induced by kernel spaces and the reducibility of the generalized Cowen-Douglas operator.  By using the geometric approach, we study the reducing subspaces of $S_z$ on certain polynomial quotient modules.
\end{abstract}


\maketitle

\section{Introduction}

Let $\mathbb{D}^2=\{(z,w): |z|<1, |w|<1 \}$ be the  open unit bidisk in $\mathbb{C}^2$, and  $\mathbb{T}^2=\{(z,w): |z|=1, |w|=1 \}$ be the distinguished boundary of $\mathbb{D}^2$.  Let $|dz|$ denote the normalized Lebesgue measure on the unit circle $\mathbb{T}$, and $dm=|dz||dw|$ be the product measure on the torus $\mathbb{T}^2$. The Hardy space $H^2(\mathbb{D}^2)$ over the bidisk is the Hilbert space of all holomorphic functions on $\mathbb{D}^2$  satisfying
\[ \|f\|^2=\sup_{0\leq r<1} \int_{\mathbb{T}^2} |f(rz,rw)|^2 dm<\infty. \]
On $H^2(\mathbb{D}^2)$,  the Toeplitz operators $T_z$ and $T_w$
are unilateral shifts of infinity multiplicity. A closed subspace
$\mathcal{M}\subset H^2(\mathbb{D}^2)$ is called a submodule if it is
invariant under the action of both $T_z$ and $T_w$.  The quotient space
$H^2(\mathbb{D}^2)\big/\mathcal{M}$,  naturally identified with
$\mathcal{N}=H^2(\mathbb{D}^2)\ominus \mathcal{M}$,   is endowed with a
$\mathbb{C}[z,w]$-module structure by
$$p\cdot f=p(S_z, S_w)f, \quad p\in \mathbb{C}[z,w], \quad f\in \mathcal{N},$$
where $S_z=P_NT_z|_\mathcal{N}$ and $S_w=P_NT_w|_\mathcal{N}$ are compressed shift operators on $\mathcal{N}$,  $\mathbb{C}[z,w]$  denotes the polynomial ring in two variables,    and $P_\mathcal{N}$ stands for the
orthogonal projection from $H^2(\mathbb{D}^2)$ onto $\mathcal{N}$.  In functional model theory, the compression  of  shift operators on the quotient modules serve as canonical model for a large class of operators. For example,  the  Bergman shift  is unitarily equivalent to $S_z$ on the quotient module $H^2(\mathbb{D}^2)\ominus[z-w]$. This idea was used successfully in studying the invariant subspaces and reducing subspaces of multiplication operators on the Bergman space \cite{GH,Ric}.

The classical Beurling theorem \cite{Beu} says that every invariant subspace of $H^2(\mathbb{D})$ has the form $\theta H^2(\mathbb{D})$ for some inner function $\theta(z)$ on $\mathbb{D}$. The model space $\mathcal{K}_\theta=H^2(\mathbb{D})\ominus \theta H^2(\mathbb{D})$ and the compressed shift operator on $\mathcal{K}_\theta$  have  played  important roles in developing both function and operator theory over the past century \cite{Ber,Nik}.  However, the structure of submodules of $H^2(\mathbb{D}^2)$ is far more complex. For instance, W. Rudin displayed two pathological submodules in \cite{Rud}, one is of infinite rank, and the another one contains no nonzero bounded holomorphic functions.  In the past few decades, many researchers  have paid attention to both the operator theory and function theory on $H^{2}(\mathbb{D}^{2})$. We refer the readers to see \cite{BL,CG,Rud,GW,Yan2001,Yan2002,Yan2019} and references  therein.

 A function $\theta \in H^{2}(\mathbb{D}^{2})$ is called  an inner function on $\mathbb{D}^2$ if  $|\theta(e^{i\theta_1},e^{i\theta_2})|=1$ almost everywhere on $\mathbb{T}^{2}$. To every inner function $\theta(z,w)$ on $\mathbb{D}^2$, the associated quotient module
\[\mathcal{K}_{\theta}=H^2(\mathbb{D}^2)\ominus \theta H^2(\mathbb{D}^2)\]
 is called the Beurling type quotient module in the frame of Hilbert modules \cite{DP}. The compressed shift operators $S_z$ and their adjoints $S_z^*$  on the Beurling type quotient modules  $\mathcal{K}_{\theta}$ form a rich area of analysis. For instance,   the structure of $\mathcal{K}_{\theta}$ induced by rational inner function is related to the Agler decomposition and stable polynomials (see \cite{BK,BL,Kne}), and the techniques from algebraic geometry and complex geometry are intrinsic to the study of $S_z$.

 Let $T$ be a bounded linear operator on a Hilbert space $\mathcal{H}$, and a closed subspace $\mathcal{M}$ of $\mathcal{H}$ is called a reducing subspace of $T$  if $T\mathcal{M}\subset \mathcal{M}$  and $T^*\mathcal{M}\subset \mathcal{M}$. A reducing subspace $\mathcal{M}$ is called minimal if the only reducing subspaces contained in $\mathcal{M}$ are $\mathcal{M}$ and $\{0\}$. If $T$ has a proper reducing subspace, we say  $T$ is reducible, otherwise, $T$ is said to be irreducible. The classification of invariant subspaces and reducing subspaces for various classes of linear operators has inspired much deep research and prompted many interesting problems. Not only has the problem itself turned out to be important, but also the methods used to solve it are interesting.  Insights on the reducing subspaces of multiplication operators on the Bergman space can be found in \cite{DPW,DSZ,GH,Zhu}, and insights on the reducing subspace of truncated Toeplitz operators can be found in \cite{DF} and \cite{Li}.


In 1990, Agler \cite{Agler} showed that every $\varphi$ in the unit ball of $H^\infty(\mathbb{D}^2)$ admits the following decomposition:
\begin{equation*}
  \frac{1-\overline{\varphi(\lambda,\mu)}\varphi(z,w)}{(1-\bar{\lambda}z)(1-\bar{\mu}w)}=\frac{K_1(z,w,\lambda,\mu)}{1-\bar{\lambda}z}+\frac{K_2(z,w,\lambda,\mu)}{1-\bar{\mu}w},
\end{equation*}
where $K_1,K_2:\mathbb{D}^2\times \mathbb{D}^2\rightarrow \mathbb{C}$ are two positive kernels. For an inner function $\theta$, we see that $\mathcal{K}_{\theta}$ can be decomposed (see \cite{BSV}) as
\[ \mathcal{K}_{\theta}=\mathcal{S}_1\oplus \mathcal{S}_2, \]
where $\mathcal{S}_1=\mathcal{H}(\frac{K_1}{1-\bar{\lambda}z})$ and $\mathcal{S}_2=\mathcal{H}(\frac{K_2}{1-\bar{\mu}w})$ are $T_z$-invariant and $T_w$-invariant, respectively, and $\mathcal{H}(K)$ denotes the reproducing kernel Hilbert space with reproducing kernel $K$. These types of subspaces $\mathcal{S}_1$ and $\mathcal{S}_2$ are called \textit{Agler subspaces} of $\mathcal{K}_{\theta}$. If $\mathcal{S}_1$ is also a reducing subspace of $S_z$, we call it \textit{Agler reducing subspace} of $S_z$. We say $S_z$ is \textit{Agler reducible} if it has a nontrivial Agler reducing subspace.

In 2017, for a rational inner function $\theta$, K. Bickel, G. Knese and C. Liaw studied $\mathcal{K}_{\theta}$ by Agler decomposition in \cite{BK,BL} and the  Agler  reducibility  of $S_{z}$    in \cite{BL}. They showed that $S_z$ on the Beurling-type quotient module $\mathcal{K}_{\theta}$ induced by a rational inner function is Agler reducible if and only if $\theta$ is the product of two one-variable inner functions. In \cite{ZYL}, we extended this result to any inner functions and defined \textit{pure isometry reducing subspaces}. An isometry $T$ on Hilbert space $\mathcal{H}$ is called a pure isometry if $\bigcap_{n=0}^\infty T^n \mathcal{H}=\{0\}$. A reducing subspace $\mathcal{M}$ for $T$ is said to be a pure isometry reducing subspace if $T|_{\mathcal{M}}$ is a pure isometry. Since the subspace $\mathcal{S}_1$ is $T_z$-invariant, thus $S_z|_{\mathcal{S}_1}=T_z|_{\mathcal{S}_1}$ is a pure isometry, and hence each Agler reducing subspace $\mathcal{S}_1$ is a pure isometry reducing subspace for $S_z$. However, there  exists inner function $\theta$ such that $S_z$ on $\mathcal{K}_{\theta}$ has nontrivial pure isometry reducing subspaces but no nontrivial Agler reducing subspaces. Further, we proved that for an inner function $\theta$, $S_z$ on $\mathcal{K}_{\theta}$ has a nontrivial pure isometry reducing subspace if and only if $\theta$ has a nonconstant one-variable inner factor $\psi(w)$ (see Theorem 2.1, \cite{ZYL}). It is worth noting that there are many inner functions such that $S_z$ on $\mathcal{K}_{\theta}$ is reducible but has no nontrivial pure isometry reducing subspaces. This leads to the following question.

\begin{question}\label{q2}
For an inner function $\theta$ on $\mathbb{D}^2$, when is   $S_{z}$ reducible on  $\mathcal{K}_{\theta}$?
\end{question}

For a polynomial $p(z,w)\in \mathbb{C}[z,w]$, let $deg_{z} ~p$, $deg_{w} ~p$ denote the degrees of $p$ in variable $z$ and $w$, respectively. Let $r(z,w)=\frac{q(z,w)}{p(z,w)}$ be a rational function in two variables, where $p$ and $q$ have no common factors. Define $deg_{z}~ r=\max\{deg_{z}~ p, deg_{z} ~q\}$, $deg_{w} ~r=\max\{deg_{w} ~p, deg_{w} ~q\}$, and define the degree of $r(z,w)$ by the tuple $(deg_{z} ~r, deg_{w} ~r)$. A rational inner function with degree $(m,n)$ is called as $(m,n)$-type. In \cite{YZL,ZYL}, we observe that the characteristic function of $S_z$ relates closely to analytic Toeplitz operators and then introduce the notion of a weakly reducing subspace for an operator valued function $\Theta(z)$. By establishing the relationship of the weak reducibility of the characteristic function and the reducibility of $S_z$,  we gave a complete characterization of  the reducibility of $S_z$ on $\mathcal{K}_{\theta}$ for $(n,1)$-type rational inner functions. However, this method is not suitable for the rational inner functions with high-degree. Even for the quotient module $\mathcal{K}_{\theta}$  induced by a $(2,2)$-type rational inner function, the classification of the reducing subspaces of $S_z$ is not fully understood. The present paper is a continuation of \cite{YZL,ZYL} and a series of  related works, such as \cite{BL,ChD,CoD,Yan2002,Yan2019}. Our aim is to study the reducibility of  the compressed shift operator $S_z$ on the Beurling type quotient modules  $\mathcal{K}_{\theta}$ induced by rational inner functions as well as on certain polynomial quotient modules via complex geometry approach.

Firstly, we calculate the spectrum and essential spectrum of $S_z$ on $\mathcal{K}_{\theta}$ and note that $S_z^*$ has rich point spectrum. The rich point spectrum structure of $S_z^*$ inspires us to promote the following question.
\begin{question}\label{question-CD}
When is $S_z^*$  a Cowen-Douglas operator on the Beurling type quotient modules $\mathcal{K}_{\theta}$?
\end{question}
In \cite{CoD}, it was shown that the reducibility of a Cowen-Douglas operator is equivalent to the reducibility of the associated kernel space vector bundle. In particular, if the kernel space vector bundle  is a line bundle, then the Cowen-Douglas operator is irreducible. In section 3, we prove that for $\mathcal{K}_{\theta}$ induced by  a rational inner function $\theta$ in $\mathbb{D}^2$, $S_z^*$ is in the Cowen-Douglas class if and only if the projection  of the zero set of every irreducible factor of $\theta$ onto the $z$-plane is connected in $\mathbb{D}$.  Moreover, we provide two interesting examples such that $S_z^*$ fail to be in Cowen-Douglas class, but one is irreducible (see Example \ref{counterexample}) and the another one  is reducible (see Example \ref{counterexample2}). By a further detailed observation,  we find that for general Beurling type quotient module  $\mathcal{K}_{\theta}$, compare to the vector bundle induced by a Cowen-Douglas operator, the kernel space vector bundle $E_{S_z^*}$ is always  a union of some Hermitian holomorphic vector bundles. Therefore by a slightly modification of the original definition of Cowen-Douglas operators,  we define  the generalized Cowen-Douglas operators $\mathfrak{B}_{\alpha}(\Omega)$ , which will be used for us to  study the reducibility of $S_z$.

\begin{definition}\label{general-definition}
Let $\Omega$ be a bounded open subset (not necessary to be connected) of $\mathbb{C}$ and $\mathcal{H}$ be a separable complex Hilbert space. Suppose
\[\Omega=\bigcup_{i=1}^N \Omega_i,\]
 for some $N\in \mathbb{Z}_+\bigcup \{\infty\}$  and  $\Omega_i$ are the distinct connected components of $\Omega$. Let $\alpha=(\alpha_i)_{i=1}^{N}$ be a sequence of positive integers.
The  generalized Cowen-Douglas class  $\mathfrak{B}_{\alpha}(\Omega)$ consists of the bounded liner operators $T$ on $\mathcal{H}$ with the following properties:
\begin{enumerate}

  \item $Ran (T-w)=\mathcal{H}$ for $w\in \Omega$;
  \item  $\bigvee_{w\in \Omega} Ker (T-w) = \mathcal{H}$;
  \item  $\dim Ker(T-w)=\alpha_i$ for any  $w\in \Omega_i$ and all $i$.
\end{enumerate}
In particular, if $N=1$, then the definition is exactly the  Cowen-Douglas operator class.
\end{definition}

For  an operator $T$ in  $\mathfrak{B}_{\alpha}(\Omega)$ and a connected component $\Omega_i$ of $\Omega$, let $(E_T(\Omega_i), \pi_i)$ denote the subbundle of the trivial bundle $\Omega_i \times \mathcal{H}$ defined by
\[ E_T(\Omega_i)=\left\{(w,x)\in \Omega_i \times \mathcal{H} :   x\in Ker(T-w)\right\}\]
 and $\pi_i(w,x)=w.$
Then $E_T(\Omega_i)$ is Hermitian holomorphic vector bundle over $\Omega_i$ of dimension $\alpha_i$ (cf. \cite{CoD}), and we call it the kernel space vector bundle associated with operator $T$. Let $(E_T, \pi)$ be the  disjoint union of all subbundles $(E_T(\Omega_i), \pi_i)$, denoted by
\[ E_T= \bigsqcup_{i=1}^N E_T(\Omega_i). \]

For $\mathcal{K}_{\theta}$ induced by a rational inner function, we  show that $S_z^*$ is a generalized Cowen-Douglas operator  and establish the connection between the reducibility of $S_z$ and the reducibility of the kernel space vector bundle $E_{S_z^*}$.
 \begin{Thm}[Theorem \ref{4.7} and Theorem \ref{strictly-reducible}]
Suppose $\theta(z,w)$ is a rational inner function  without  univariate inner factors and   $\mathcal{K}_{\theta}=H^2(\mathbb{D}^2)\ominus \theta H^2(\mathbb{D}^2)$ is the corresponding Beurling type quotient module.  Let
 \[\Omega=\sigma(S_z)\setminus \sigma_e(S_z)=\bigcup_{i=1}^N \Omega_i,\]
 where $\Omega_i, i=1,\ldots, N$ are the connected components of $\Omega$,
 then $S_z^*$ belongs to $\mathfrak{B}_{\alpha}(\Omega^*)$ with index $\alpha=\{\alpha_i\}$, where $\alpha_i=Z_\theta|_{\Omega_i}$ is the numbers of zeros of $\theta(\lambda, \cdot)$ in $\mathbb{D}$, $\lambda\in \Omega_i$.  $S_z$ is reducible if and only if the kernel space vector bundle $E_{S_z^*}$ is strictly   reducible (the definition of ``strictly   reducible" see section 4 for more details).
\end{Thm}

The reminder of the paper is organized as follows. 
In Section 2, we   study the   spectrum and the essential spectrum of $S_z^*$ for the Beurling type quotient module. In Section 3, we   introduce the basic properties of the Cowen-Douglas operator and completely characterize  when is $S_z^*$    a Cowen-Douglas operator. In Section 4, we extend  the definition of the Cowen-Douglas operators to the generalized Cowen-Douglas operators.   Some elementary   geometric  properties and   applications are obtained. We show that $S_z^*$ on $\mathcal{K}_{\theta}$ induced by a rational inner function is  a generalized Cowen-Douglas operator  and  $S_z$ is reducible if and only if $E_{S_z^*}$ is strictly   reducible. In section 5, we define the divisor property for polynomials and  generalize our results to polynomial quotient modules, we conclude that $S_z^*$ on general polynomial quotient module is also a generalized Cowen-Douglas operator and $S_z$ is reducible if and only if $E_{S_z^*}$ is strictly   reducible. Moreover, by using the complex geometric approach,  we  study the reducibility of the compressed shift operators on  polynomial quotient modules $[z^m-w^n]^\perp$ and $[(z-w)^n]^\perp$.

\section{The spectrum of the compressed shift $S_z$ }
In this section, we  will study the spectrum of $S_z$ on the Beurling type quotient module $\mathcal{K}_{\theta}=H^2(\mathbb{D}^2)\ominus \theta H^2(\mathbb{D}^2)$.    To describe the spectrum of $S_z$, we first introduce the characteristic function of a completely nonunitary (c.n.u.) contractive operator.

 For a completely nonunitary contraction $T$ on the Hilbert space $\mathcal{H}$, one can associate with it two defect operators $D_T=(I-T^*T)^{1/2}$, $D_{T^*}=(I-TT^*)^{1/2}$, and two defect spaces $\mathcal{D}_T$ and $\mathcal{D}_{T^*}$ which are the closure of the range $D_T$ and $D_{T^*}$ respectively. The operator-valued analytic function
\[ \Theta_{T}(\lambda) =[-T+\lambda D_{T^*}(1-\lambda T^*)^{-1} D_T]|_{\mathcal{D}_T},~~~~\lambda \in \mathbb{D} \]
is called the  characteristic  function  for $T$. In functional model theory,  the defect operators and the characteristic  functions are very useful tools in determining the structure of contractions (cf. \cite{NF}).

In the following, $H_{w}^2$ will denote  the Hardy space in variable $w$, which can be regarded as a subspace of $H^2(\mathbb{D}^2)$.   For $\lambda \in \mathbb{D}$, the left evaluation operator  $L(\lambda):H^2(\mathbb{D}^2)\rightarrow H_{w}^2$ is  defined by
\[L(\lambda)h(z,w)=h(\lambda,w), ~~h\in H^2(\mathbb{D}^2)\]
The restriction of $L(\lambda)$  to a closed subspace $K\subset H^2(\mathbb{D}^2)$ will be denoted by $L_K(\lambda)$. By \cite{Yan2002}, there are constant unitaries $U,V,W$ such that
\begin{equation}\label{2.1}
  L(\lambda)|_{M\ominus zM}=(U\Theta_{S_z}(\lambda)V)\oplus W.
\end{equation}
 It is not hard to see  that $W=0$ if and only if $M$ (or equivalently, $M\ominus zM$) does not contain nonzero functions which  depend only on   $w$. In particular, if $M=\theta H^2(\mathbb{D}^2)$ for some inner function $\theta$, then $W=0$ if and only if $\theta(z,w)$ is not an inner function depending only on $w$.

 Recalled that a operator $T$ is called a Fredholm operator if $Ran T$ is closed and $Ker T, Ker T^*$ are both finite dimensional. For a Fredholm operator $T$, we define the index of $T$ as $ind(T)=\dim Ker T-\dim Ker T^*$. For a point $\lambda$, it said to be essential spectrum point of $T$ if $T-\lambda$ is not a Fredholm operator. The set consisting of all essential spectrum point of $T$ is called as the essential spectrum $\sigma_e(T)$ of $T$.    The following lemma is crucial for us to study the spectrum and essential spectrum of $S_z$.
\begin{lemma}[\cite{DY}, Corollary 2.3]\label{Yang}
If $M$ is a $z$-invariant subspace of $H^2(\mathbb{D}^2)$, then
\[\sigma(S_{z})=\sigma (L),\]
where $\sigma (L)$ is the set of points $\lambda\in \mathbb{D}$ for which the left evaluation operator $L(\lambda)$ is not bounded invertible from $M\ominus zM$ to $H^2_{w},$ together with those $\lambda \in \mathbb{T}$ not lying on any of the open arcs of $\mathbb{T}$ on which $L(\lambda)$ is a unitary operator valued analytic function of $\lambda$.

Moreover, for $\lambda\in \mathbb{D}$, $S_z-\lambda$ is Fredholm if and only if $L_{M\ominus zM}(\lambda): M\ominus zM \longrightarrow H^2_w$ is Fredholm, and in this case, $ind(S_z-\lambda)=ind(L_{M\ominus zM}(\lambda))$.
\end{lemma}

\begin{lemma}\label{circle+contain+spectrum}
Let   $\mathcal{K}_{\theta}=H^2(\mathbb{D}^2)\ominus \theta H^2(\mathbb{D}^2)$ be the Beurling type quotient module.  If $\theta(z,w)$  is  an inner function not depending only on  variable $z$, then
\[\mathbb{T}\subset \sigma(S_z).\]
\end{lemma}
\begin{proof}
Since $\theta$ is an inner function, it is easy to see that for the submodule $M=\theta H^2(\mathbb{D}^2)$, $M\ominus zM=\theta H^2_{w}.$  If there exists $\lambda_0\in \mathbb{T}$ such that $\lambda_0 \notin \sigma(S_z)$, then by Lemma \ref{Yang}, there is an open arc of $\lambda_0$, say $V_{\lambda_0}$, such that the operator-valued function  $L_{M\ominus zM}(\lambda)$ has an unitary analytic extension on $V_{\lambda_0}$.

Define two unitary operators as follows:
\begin{equation}\nonumber
\begin{split}
 U :&~\theta H^{2}_{w} \rightarrow H^{2}_{w}\\
    &\theta(z,w) f(w) \mapsto f(w),
    \end{split}
\end{equation}
and $V: H^{2}_{w} \rightarrow H^{2}_{w} $ be the identity,   we have the following diagram:
\begin{equation*}
\CD
             \theta  H_w^2 @>L(\lambda)>> H_w^2 \\
             @V   \textrm{U} VV @V \textrm{V} VV  \\
             H_w^2 @>T_{\theta(\lambda,w)}>> H_w^2.
\endCD
\end{equation*}
It is not hard to check that
\begin{equation}\label{Toeplitz}
  T_{\theta(\lambda, w)} U =V L(\lambda)
\end{equation}

where $T_{\theta(\lambda,w)}$ be the analytic Toeplitz operators on $H^{2}_{w}$,
 which implies $T_{\theta(\lambda,w)}$ also has an unitary analytic extension on $V_{\lambda_0}$.  By the properties of Toeplitz operator, the boundary value of the symbol function $\theta(\lambda,w)$ must be a unimodular constant (depends on $\lambda$).
 Combining with the Fatou's lemma(see \cite{Gar}), we obtain  that the radial limits $\lim\limits_{r\rightarrow 1} \theta (r\lambda,w)=c(\lambda) $ with $|c(\lambda)|=1$ for almost all $\lambda \in V_{\lambda_0}$.  If we write $\theta(z,w)=\sum\limits_{n\geq0} a_n(z)w^n$ ,  then it implies that the boundary value of $a_n(z)$ vanishes almost everywhere on $V_\lambda$ for all $n>0$. As $a_n(z)\in H^2(\mathbb{D})$  and the Lebesgue measure of $V_\lambda$ is positive, it must have $a_n(z)\equiv0$ for all $n>0$. This is a contradiction, which proves that $\mathbb{T}\subset \sigma(S_z)$.
\end{proof}


A rational inner function is an inner function which is also a rational function.  By \cite{Rud} (W. Rudin, Theorem 5.2.5), every  rational inner function $\theta$ in $\mathbb{D}^2$ has the form
\[\theta(z,w)=z^k w^l \frac{\widetilde{p}(z,w)}{p(z,w)},\]
where $p$ is a polynomial with no zero in $\mathbb{D}^2$, and $\widetilde{p}$ is the reflection of $p$ with respect to $\mathbb{T}^2$, which is defined by
\[\widetilde{p}(z,w)= z^m w^n \overline{p(1/\bar{z}, 1/{\bar{w}})},\]
where  $\deg p=(m,n)$.
Since two polynomials $p, q$ are coprime if and only if the common zero set of $p,q$ is finite (see \cite{Bli}) and $|p|=|\widetilde{p}|$ on $\mathbb{T}^2$, we can always assume that  the denominator of a rational inner function has at most finitely many zeros  on $\mathbb{T}^2$.

The following lemma shows that $p$  has no zeros on $(\mathbb{D}\times \mathbb{T})\cup (\mathbb{T}\times \mathbb{D})$.
\begin{lemma}[\cite{Kne}, Lemma 3.5]
Suppose $q\in \mathbb{C}[z,w]$ has no zeros on $\mathbb{D}^2$. If $q(z_0,w_0)=0$ for some $(z_0,w_0)\in \mathbb{T}\times \mathbb{D}$, then $q(z_0,w)=0$ for all $w\in \mathbb{C}$, i.e. $z-z_0$ divides $q$. In particular, there can  be  only finitely many $z_0\in \mathbb{T}$ such that $q(z_0, \cdot)$ has a zero in $\mathbb{D}$.
\end{lemma}

In conclusion, every rational inner function $\theta$ in $\mathbb{D}^2$  has the form
\[ \theta =z^k w^l \frac{\widetilde{p}}{p}, \]
where $p$ has no zeros on $\overline{\mathbb{D}^2}\setminus \mathbb{T}^2$ and at most finitely many zeros on $\mathbb{T}^2$.
For a holomorphic function $f$ on a domain $\Omega$,  we denote its zero set in $\Omega$  by $Z_f$, and for a rational function $\theta=q/p$ ($p,q$ are coprime), $Z_\theta$ means $Z_q$ in $\mathbb{C}^2$.
\begin{theorem}\label{spec}
Suppose $\mathcal{K}_{\theta}=H^2(\mathbb{D}^2)\ominus \theta H^2(\mathbb{D}^2)$ is a Beurling type quotient module. If $\theta(z,w)$ is an inner function not depending only on variable $z$, then
\[ \overline{\pi_1(Z_\theta)}\cup\mathbb{T} \subseteq \sigma (S_z),\]
where  $\pi_1$ is the component projection defined by $\pi_1(z,w)=z$. Moreover, if $\theta$ is a rational inner function, then
\[\sigma(S_z)=\pi_1(Z_\theta \cap  \overline{\mathbb{D}^2}). \]
\end{theorem}
\begin{proof}

For $\lambda\in \mathbb{D}$,  by (\ref{2.1}) and  (\ref{Toeplitz}), we see that  $S_z-\lambda$ is not invertible if and only if $T_{\theta(\lambda, \cdot)}$ is not invertible, this is equivalent to $0\in \overline{range~\theta(\lambda,\cdot)}$, hence $\pi_1(Z_\theta)\subseteq \sigma(S_z)$ and this proves the first part of the theorem.

Now suppose that  $\theta=z^k w^l \frac{\widetilde{p}}{p}$ is a  rational inner function. Since $p$ has no zeros on $\overline{\mathbb{D}^2}\setminus \mathbb{T}^2$, $\theta(\lambda,\cdot)$ belongs to the disk algebra for any $\lambda\in \mathbb{D}$, hence $\overline{range~\theta(\lambda,\cdot)}= range~ \theta(\lambda, \cdot)|_{\overline{\mathbb{D}}}$. So for $\lambda\in \mathbb{D}$, $S_z-\lambda$ is not invertible if and only if $\theta(\lambda, \cdot)$ has zeros in $\overline{\mathbb{D}}$, i.e.
  \[ \mathbb{D}\cap \sigma(S_z)=\pi_1(Z_\theta \cap \mathbb{D}\times \overline{\mathbb{D}}).\]

  For $\lambda\in \mathbb{T}$, if $\theta(z,w)$ is  a rational inner function not depending only on variable  $z$, write $\theta(z)=B(z) \frac{\widetilde{p}}{p}$, where $p$ has no factors depending only on variable $z$ and $B(z)$ is a finite Blaschke product, as $p$ has no zeros on $\mathbb{T}\times \mathbb{D}$, we obtain that for every $\lambda\in \mathbb{T}$,  $p(\lambda,\cdot)$ has no zeros in $\mathbb{D}$, thus $\widetilde{p}(\lambda,\cdot)$ must have zeros in $\overline{\mathbb{D}}$, which means that  $$\mathbb{T}= \pi_1(Z_\theta \cap \mathbb{T}\times \overline{\mathbb{D}}).$$
  Then combine Lemma \ref{circle+contain+spectrum}, we obtain  the desired conclusion. On the other hand, if $\theta(z,w)=B(z)$ is finite Blaschke product depending only on $z$, then $M\ominus zM=B(z)H_w^2$, and therefore $L_{M\ominus zM}(\lambda)$ has a natural unitary analytic extension on $\mathbb{T}$, hence $\mathbb{T}\subseteq \rho(S_z)$, where $\rho(S_z)$ denote the resolvent set of $S_z$. So we also have
  \[\sigma(S_z)=\pi_1(Z_\theta \cap  \overline{\mathbb{D}^2} ).\]
\end{proof}

As one known, the inner-outer factorization does not always exist for the function in $H^2(\mathbb{D}^n) (n>1)$, that makes many useful   techniques in $H^2(\mathbb{D})$ relating  to the inner-outer factorization cannot generalize to the multivariable Hardy space.  However, the role of the Blaschke products is taken over by the good inner functions(\cite{Rud}).

An inner function $g$ in $\mathbb{D}^n$ is said to be good if $u[g]=0$, where $u[g]$ is the least $n$-harmonic majorant of $\log |g|$ in $\mathbb{D}^n$.
 For $n=1$, the good inner functions are precisely the Blaschke products.
If $f_1$ and $f_2$ are holomorphic functions in $\mathbb{D}^n$, the statement ``$f_1$ and $f_2$ have same zeros" will mean: ``$f_1=hf_2$, $h$ is holomorphic in $\mathbb{D}^n$, and $h$ has no zero in $\mathbb{D}^n$", i.e., multiplicities are taken into account. In this sense, the good inner functions are determined by their zeros in $\mathbb{D}^n$ (see \cite{Rud}).   In particular, a  rational inner function is  a good inner function.

The following factorization theorem will be used frequently.
\begin{lemma}[\cite{Rud},  Theorem 5.4.2] \label{Rudin}
 If $0<p<\infty$ and $f\in H^p(\mathbb{D}^n)$, $g$ is a good inner function, $h$ is holomorphic in $\mathbb{D}^n$, and $f=gh$, then $h\in H^p(\mathbb{D}^n)$, and $\|h\|_p=\|f\|_p$.
\end{lemma}

In order to study the essential spectrum of $S_z$,  we first consider the kernel space and cokernel space of $S_{z-\lambda}$. For convenience, we use $\Omega^*$ to denote the complex conjugate of a set $\Omega$.
\begin{theorem}\label{kernel}
Suppose $\mathcal{K}_{\theta}=H^2(\mathbb{D}^2)\ominus \theta H^2(\mathbb{D}^2)$ is a Beurling type quotient module.  Assume that $\theta(z,w)=B(z) \phi(z,w)$ with $B(z)=\prod\limits_{i=1}^\infty \frac{z-\lambda_i}{1-\overline{\lambda_i}z}$  and $\phi(\lambda,\cdot) $ being not identical zero for any $\lambda\in \mathbb{D}$,  then the followings hold.

\begin{enumerate}
\setlength{\itemsep}{0pt}
\item  For   $\lambda\in \mathbb{T}$,  then
\[ Ker S_{z-\lambda} =Ker S_{z-\lambda}^*=\{0\}.\]

\item For $\lambda \in \mathbb{D}$,~~ if $B(\lambda)\neq 0$, then $Ker S_{z-\lambda}=\{0\}$, and  if $B(\lambda)=0$, then
\[Ker S_{z-\lambda}=\left\{\frac{\theta(z,w)g(w)}{z-\lambda} :~ g(w)\in H_w^2 \right\}.\]

\item  For $\lambda \in \mathbb{D}$, \[Ker S_{z-\lambda}^*=\left\{\frac{g(w)}{1-\overline{\lambda}z} : g(w)\in H_w^2\ominus \theta(\lambda,w)H^2_w\right\}.\]
\end{enumerate}
In particular, $\sigma_p(S_z)=\{ \lambda_i : i\geq 1\}$, $\pi_1(Z_\theta \cap \mathbb{D}^2)^* \subseteq \sigma_p(S_z^*)$.
\end{theorem}
\begin{proof}
(1)   Let  $f\in \mathcal{K}_{\theta}$ such that $f\in Ker S_{z-\lambda}$, then
 $$(z-\lambda)f(z,w)=\theta(z,w) g(z,w)$$
 for some $g\in H^2(\mathbb{D}^2)$.

 If $\lambda\in \mathbb{T}$, then there is  a sequence of polynomials   $p_n(z)$ such that
   $$(z-\lambda)p_n \xrightarrow{\|\cdot\|_2} 1,$$
   hence $\theta g p_n \xrightarrow{\|\cdot\|_1} f$. Note that  $\theta$ is an inner function, $\{ g p_n \}$ is  a Cauchy sequence in $H^1(\mathbb{D}^2)$, and assume that  $ g p_n \xrightarrow{\|\cdot\|_1} F$, then $f=\theta F$. Since $f\in H^2(\mathbb{D}^2)$, $F$ must be in $H^2(\mathbb{D}^2)$, thus $f\in \theta H^2(\mathbb{D}^2)$, which implies  that $f=0$. This proves the first part of (1) and the second part of (1) will be contained in the proof of (3).

(2) Suppose $\lambda\in \mathbb{D}$.  If $B(\lambda)\neq 0$, then $\theta(\lambda,w)$ is not identical zero,  hence $g(\lambda, w)\equiv 0$.   Since the inner factor $\varphi_\lambda=\frac{z-\lambda}{1-\overline{\lambda}z}$ is a good inner function, by Lemma \ref{Rudin},  $h=g/\varphi_\lambda \in H^2(\mathbb{D}^2)$. It follows that    $f=\theta \frac{h}{1-\overline{\lambda}z} \in \theta H^2(\mathbb{D}^2)$, and  we  must have $f=0$.

If $B(\lambda)=0$, then $f=\frac{1}{1-\overline{\lambda}z} \frac{\theta}{\varphi_\lambda}g$.  Note that
\[\langle\frac{1}{1-\overline{\lambda}z} ~\frac{\theta(z,w)}{\varphi_\lambda(z)}~g(\lambda,w), \theta(z,w)h(z,w) \rangle= \langle \frac{1}{1-\overline{\lambda}z} g(\lambda,w), \varphi_\lambda(z) h(z,w) \rangle=0\]
for any $h\in H^2(\mathbb{D}^2)$,  it follows that
\begin{align*}
  0=\langle f, \theta h \rangle & =\langle \theta\frac{g(z,w)-g(\lambda,w)}{z-\lambda}, \theta h \rangle \\
   & =\langle \frac{g(z,w)-g(\lambda,w)}{z-\lambda}, h \rangle
\end{align*}
for any $h\in H^2(\mathbb{D}^2)$.  It implies $g(z,w)\equiv g(\lambda,w)$ and therefore
\[Ker S_{z-\lambda} \subseteq\left\{\frac{\theta(z,w)g(w)}{z-\lambda}: g(w)\in H_w^2\right\}.\]
It is easy to check the opposite inclusion
\[ \left\{\frac{\theta(z,w)g(w)}{z-\lambda}: g(w)\in H_w^2 \right\}\subseteq Ker S_{z-\lambda},\]
and this proves (2).

(3) Let  $f\in Ker S_{z-\lambda}^*$,  since $S_z^*f=\frac{f(z,w)-f(0,w)}{z}$, we have
\[f(z,w)=\frac{f(0,w)}{1-\overline{\lambda}z}.\]
For $\lambda \in \mathbb{T}$, $\frac{f(0,w)}{1-\overline{\lambda}z}$ is square integrable if and only if $f(0,w)\equiv 0$, hence $f=0$ and this proves second part of (1). For $\lambda\in \mathbb{D}$,  it is not hard to see that $\frac{f(0,w)}{1-\overline{\lambda}z}\in \mathcal{K}_{\theta}$ if and only if $f(0,w) \perp \theta(\lambda,w)H_w^2$ and this proves (3).
\end{proof}

Combine Theorem \ref{spec} and Theorem \ref{kernel}, we have the following theorem.
\begin{theorem}\label{spec2}
Suppose $\mathcal{K}_{\theta}=H^2(\mathbb{D}^2)\ominus \theta H^2(\mathbb{D}^2)$ is a Beurling type quotient module. If $\theta$ has no factor depending only on $z$,  then

\[ \sigma_p(S_z)=\emptyset, ~~~~\mathbb{T}\subseteq  \sigma_e(S_z). \]
Moreover, if $\theta$ is a rational inner function, then
\[ \sigma_p(S_z^*)=\pi_1(Z_\theta \cap \mathbb{D}^2)^*\]
\[ \sigma_e(S_z)=\mathbb{T}\cup \pi_1(Z_\theta \cap \mathbb{D}\times \mathbb{T})=\pi_1(Z_\theta \cap \partial \mathbb{D}^2).\]
For $\lambda\in \sigma(S_z)\setminus \sigma_e(S_z) $, $S_{z-\lambda}^*$ is Fredholm with index
\[ ind(S_{z-\lambda}^*)=Z_\theta(\lambda), \]
where $Z_\theta(\lambda)$ denotes the number of zeros of $\theta(\lambda,\cdot)$ in $\mathbb{D}$.
\end{theorem}
\begin{proof}
By Theorem \ref{spec}, $\mathbb{T}$ is always contained in $\sigma(S_z)$. Combine with the fact that $Ker S_{z-\lambda} =Ker S_{z-\lambda}^*=\{0\}$ for $\lambda\in \mathbb{T}$, we conclude that $Ran S_{z-\lambda}$ is not closed, otherwise $S_{z-\lambda}$ will be invertible. Therefore
$\mathbb{T}\subset \sigma_e(S_z)$.

 If $\theta$ is rational, for $\lambda\in \mathbb{D}$, by (\ref{2.1}) and (\ref{Toeplitz}), $Ran S_{z-\lambda}$ is  closed if and only $Ran T_{\theta(\lambda, \cdot)}$ is closed. This happens if and only if $\theta(\lambda,\cdot)$ is bounded away zero on $\mathbb{T}$, i.e. $\theta(\lambda,\cdot)$ has no zeros on $T$. Moreover, in this case, the inner factor of $\theta(\lambda,w)$ is always a finite Blaschke product, and it is uniquely determined by the zeros of $\theta(\lambda, w)$ in $\mathbb{D}$. Hence by Theorem \ref{kernel}, $Ker S_{z-\lambda}^*$ is always finite dimensional and $ind(S_{z-\lambda}^*)=Z_\theta(\lambda)$.
\end{proof}
\begin{remark}\label{remark2.8}
Let $K_\lambda(\cdot)=\frac{1}{1-\bar{\lambda}\cdot}$ be the reproducing kernel of $H^2(\mathbb{D})$,  for a rational inner function $\theta$ and $\lambda\in \sigma(S_z)\setminus \sigma_e(S_z)$, let $\{a_k(\lambda)\}$ be the zeros in $\mathbb{D}$ of $\theta(\lambda,\cdot)$ and $l_k$ be the corresponding order of $a_k(\lambda)$, then the canonical basis of $Ker S_{z-\lambda}^*$ has the form
\[\left\{ K_{\lambda}(z)\cdot K_{a_k(\lambda)}^{(j)}(w) , ~~~~k\geq 0~~and ~~j=0,1,\cdots,l_k-1\right\}.\]
Here $K^{(l)}_\zeta(w)=\frac{l! w^l}{(1-\bar{\zeta}w)^{l+1}}$.
\end{remark}

 The following example is an application of Theorem \ref{spec} and Theorem \ref{spec2}.

\begin{example}\label{(1,1)-type}
\begin{enumerate}
\item For $\theta(z,w)=\frac{zw-t}{1-tzw}, 0<t<1,$ note that $Z_\theta\cap \overline{\mathbb{D}^2}=\{~(z,\frac{t}{z}): t\leq |z|\leq1~\}$,we have
\[ \sigma(S_z) = \left\{\lambda\in \mathbb{C} : t\leq |\lambda| \leq 1  \right\},\]
\[ \sigma_e(S_z)=\mathbb{T}\cup \{\lambda\in \mathbb{C} :  |\lambda|=t \},  \]
\[ \sigma_p(S_z^*) = \left\{\lambda\in \mathbb{C} : t< |\lambda| < 1  \right\}. \]
  \item For $\theta(z,w)=\frac{zw-tz-(1-t)w}{1-tw-(1-t)z}, 0<t<1$, note that $Z_\theta\cap \overline{\mathbb{D}^2}=\{~(z,\frac{tz}{z-(1-t)}): |z|\leq1, |tz|\leq |z-(1-t)|~\}$, we have
\[  \sigma(S_z)=\overline{\mathbb{D}}\setminus\left \{\lambda\in \mathbb{C} : \left|\lambda-\frac{1}{1+t}\right|<\frac{t}{1+t}\right\}, \]
\[\sigma_e(S_z)=\mathbb{T}\cup \left\{\lambda\in \mathbb{C} : \left|\lambda-\frac{1}{1+t}\right|=\frac{t}{1+t}\right\}, \]
\[ \sigma_p(S_z^*)=\mathbb{D}\setminus\left \{\lambda\in \mathbb{C} : \left|\lambda-\frac{1}{1+t}\right|\leq \frac{t}{1+t}\right\}.  \]
\end{enumerate}

\end{example}

\section{Geometric property of the compressed shift operators}
Firstly we provide necessary preliminaries on Hermitian vector bundles, which are
mainly extracted from \cite{CoD} and \cite{ChD}. Given a domain $\Omega$ in $\mathbb{C}$, a \textit{rank $n$ holomorphic vector bundle} over $\Omega$ is a complex manifold
$E$ with a holomorphic map $\pi$ from $E$ onto $\Omega$ such that each fibre $E_\lambda=\pi^{-1}(\lambda)$ is a copy of $\mathbb{C}^n$ and for each $\lambda_0$ in $\Omega$, there exists a neighborhood $\Delta$ of $\lambda_0$ and holomorphic functions $s_1,\cdots, s_n$
from $\Delta$ to $E$ such that $E_\lambda=\bigvee \{s_1(\lambda,\cdots,s_n(\lambda))\}$. The $n$-tuple of functions $\{s_1,\cdots, s_n\}$ is called a holomorphic \textit{frame} over $\Delta$. A \textit{cross-section} is a map $s$ from an open subset of $\Omega$ to $E$ such that $\pi(s(\lambda))=\lambda$. A \textit{bundle map} is a map $\varphi$ between
two bundles $E^1$ and $E^2$ over $\Omega$ which is holomorphic, and defines a linear transformation from $E^1_\lambda$ to $E^2_\lambda$ for $\lambda$ in $\Omega$.

A \textit{Hermitian holomorphic vector bundle} $E$ over $\Omega$ is a holomorphic vector bundle such that each fiber $E_\lambda$ is an inner product space. The bundle is said to have a \textit{real-analytic (resp. $C^\infty$) metric} if $\lambda\rightarrow \|s(\lambda)\|^2$ is real-analytic (resp. $C^\infty$) for each holomorphic cross-section of $E$. Two
Hermitian holomorphic vector bundles $E^1$ and $E^2$ over $\Omega$ will be said to be equivalent if there exists an isometric holomorphic bundle map from $E^1$ onto $E^2$.

Let $\mathcal{E}(\Omega)$ denote the algebra of $C^\infty$ functions on $\Omega$ and let $\mathcal{E}^p(\Omega)$   denote the $C^\infty$ differential forms of degree $p$ on $\Omega$. Then we have $\mathcal{E}^0(\Omega)=\mathcal{E}(\Omega)$, $\mathcal{E}^1(\Omega)=\{ f dz + gd\bar{z}: f,g \in \mathcal{E}(\Omega) \}$ and $\mathcal{E}^2(\Omega)=\{f dzd\bar{z}: f \in\mathcal{E}(\Omega)\}$. For a $C^\infty$ vector bundle $E$ over $\Omega$, let $\mathcal{E}^p(\Omega,E)$ denote the differential forms of degree $p$
with coefficients in $E$, then $\mathcal{E}^0(\Omega,E)$ are just $C^\infty$  cross-sections of $E$ on $\Omega$.

A \textit{connection} on $E$ is a first order differential operator $D : \mathcal{E}^0(\Omega,E)\rightarrow \mathcal{E}^1(\Omega,E) $ such that
\[ D(f\sigma)=df\otimes \sigma +fD(\sigma) \]
for $f$ in $\mathcal{E}(\Omega)$ and $\sigma$ in $\mathcal{E}^0(\Omega,E)$. The connection $D$ is called \textit{metric-preserving} if
\[ d\langle \sigma_1, \sigma_2 \rangle=\langle D\sigma_1, \sigma_2\rangle+\langle \sigma_1, D\sigma_2 \rangle \]
for $\sigma_1,\sigma_2$ in $\mathcal{E}^0(\Omega,E)$.

Locally, $D$ can be represented by a connection matrix. Let $s =\{s_1,\cdots,s_n\}$ be a local frame on $\Delta$, then the connection matrix $\Theta(s)=[\Theta_{ij}]$ relative to the frame s is a matrix with 1-form entries $\Theta_{ij}$ defined on $\Delta$ such that
\[D(s_i)=\Sigma_{j=1}^n \Theta_{ij}\otimes s_j.\]
The connection $D$ can be extended to a differential operator from $\mathcal{E}^1(\Omega,E)$ to $\mathcal{E}^2(\Omega,E)$ so that
\[ D(\sigma\otimes \alpha)=D\sigma \wedge \alpha +\sigma\otimes d\alpha \]
for $\sigma$ in $\mathcal{E}^0(\Omega,E)$ and $\alpha$ in $\mathcal{E}^1(\Omega)$. It is well known that $D^2$ is $C^\infty$ linear,  so for any $\sigma$ in $\mathcal{E}^0(\Omega,E)$,  we have
\[ D^2 \sigma= \mathcal{K}\sigma dzd\bar{z}, \]
where $\mathcal{K}$ is a bundle map on $E$ which uniquely determined by $D^2$. Thus $D^2$ can be identified with $\mathcal{K}$ and we call $\mathcal{K}$ the \textit{curvature} of $(E,D)$.


\begin{definition}\label{definition+cowen-douglas}
Let $\Omega$ be a connected open subset of $\mathbb{C}$ and $\mathcal{H}$ be a separable Hilbert space. For a positive integer $n$, the Cowen-Douglas class  $\mathfrak{B}_n(\Omega)$  consists of the bounded linear operators $T$ on $\mathcal{H}$ with the following properties:
\begin{enumerate}
  \item $Ran (T-w)=\mathcal{H}$ for $w\in \Omega$;
  \item  $\bigvee_{w\in \Omega} Ker (T-w) = \mathcal{H}$;
  \item $\dim Ker(T-w)=n$ for any  $w\in \Omega$.
\end{enumerate}
\end{definition}

For a fixed  point $w_0\in \Omega$, the condition (2) in definition \ref{definition+cowen-douglas} is equivalent to
 \[ \bigvee_{n=1}^\infty Ker (T-w_0)^n = \mathcal{H}. \]
 If $\Omega_0$ is a connected open subset of $\Omega$, then $\mathfrak{B}_n(\Omega)\subset \mathfrak{B}_n(\Omega_0)$ because
 $\bigvee_{w\in \Omega_0} Ker (T-w)=\bigvee_{w\in \Omega} Ker (T-w)$ (\cite{CoD}, Corollary 1.13).
 Further, for $T$ in $\mathfrak{B}_n(\Omega)$, the mapping $w\mapsto Ker (T-w)$ defines a complex vector bundle $E_T$ over $\Omega$.  One   connection between $T$ and $E_T$ is contained in the following lemma.
\begin{lemma}[\cite{CoD}, Propositon 1.18]\label{reducibility}
An operator $T$ in $\mathfrak{B}_n(\Omega)$ is reducible if and only if the complex bundle $E_T$ is reducible as a Hermitian holomorphic vector bundle. In particular, an operator $T$ in $\mathfrak{B}_1(\Omega)$ is irreducible.
\end{lemma}
%


Let $E$ be a $C^\infty$ Hermitian vector bundle of rank $n$ over a domain $\Omega$ in $\mathbb{C}$ with metric-preserving connection $D$ and curvature $\mathcal{K}$. We denote by $\mathcal{A}$ the algebra of bundle maps generated by the curvature $\mathcal{K}$ and its covariant derivatives $\mathcal{K}_{z^i \bar{z}^j}$ to all orders, then the algebra $\mathcal{A}$ is self-adjoint since $\mathcal{K}$ is self-adjoint (cf.\cite{ChD}). Let $s$ be a $C^\infty$ orthonormal frame of $E$ over an open subset $\Delta$ of $\Omega$. For a bundle map $\Phi$ on $E$ and $z$ in $\Delta$, let $\Phi(z)$ be the induced fibre map on the fibre $E_z$ and $\Phi(s)(z)$ be the matrix of $\Phi(z)$ relative the base $s(z)$. We denote by $\mathcal{A}(z)$ the set of linear transforms on the fibre $E_z$ induced by the bundle map in $\mathcal{A}$ and $\mathcal{A}(s)(z)$ the matrix algebra generated by the matrices $\Phi(s)(z)$ for $\Phi$ in $\mathcal{A}$, then $\mathcal{A}(s)(z)$ is a self-adjoint matrix algebra in $M_n(\mathbb{C})$ since $s$ is orthogonal.

It is well known that any self-adjoint matrix algebra is a direct sum of full matrix algebras with multiplicity. More precisely, for any self-adjoint matrix algebra, there exist two tuples of positive integers $\textbf{m}=(m_1,\cdots,m_r)$, $\textbf{n}=(n_1,\cdots,n_r)$, such that the algebra consists of matrices of the form
\[A_1\otimes I_{m_1} \oplus A_2\otimes I_{m_2}\oplus \cdots \oplus A_r\otimes I_{m_r}, \]
where $A_i$ is an arbitrary $n_i\times n_i$ matrix repeated $m_i$ times on the diagonal, we denote such an algebra by $M(\textbf{n}, \otimes \textbf{m})$.


Now we can state the theorem on block diagonalization of the connections:
\begin{theorem}[\cite{ChD}, Theorem 4.2]\label{block-diag}

Let $E$ be a $C^\infty$ Hermitian vector bundle of rank $n$ over an open subset $\Omega\in \mathbb{C}$, with metric-preserving connection $D$. For any point $z_0$ off a non-where dense subset of $\Omega$ (the coalescing set, the set on which the dimension of a certain algebra $\mathcal{A}(z)$ is not locally constant), there exist two $r$-tuple of integers $\textbf{m}=(m_1,\cdots,m_r)$, $\textbf{n}=(n_1,\cdots,n_r)$, a neighborhood $\Omega_0$ of $z_0$ and a $C^\infty$ orthonormal frame $s$ for $E$ over $\Omega_0$ with the properties:
\[\mathcal{A}(s)(z)=M(\textbf{n}, \otimes \textbf{m})\]
for all $z$ in $\Omega_0$, where $\mathcal{A}$ is the algebra of bundle maps generated by the curvature $\mathcal{K}$, and its covariant derivatives $\mathcal{K}_{z^i \bar{z}^j}$ to all orders. Moreover,
\[\Theta(s)=\Theta_1\otimes I_{m_1} \oplus \cdots \oplus \Theta_r \otimes I_{m_r} \]
where $\Theta(s)$ is the matrix of connection 1-forms of $D$ relative to the frame $s$ and $\Theta_i$ are $C^\infty$ $n_i\times n_i$ matrices with 1-form entries defined on $\Omega_0$.
\end{theorem}

\begin{remark}
The theorem asserts that the algebra $\mathcal{A}(s)(z)$ does not depend on the point $z$ in $\Omega_0$ and the connection matrix has a block diagonal form, thus each block corresponds to a subbundle invariant under $D$. Explicitly, for any $1\leq i\leq r$, the block $\Theta_i \otimes I_{m_i}$ corresponds $m_i$ $D$-invariant sunbundles of rank $n_i$. We denote these subbundles by $E_{i1}, \cdots , E_{i m_i}$. Moreover, if $E$ is a holomorphic Hermitian vector bundle with canonical connection $D$, then $D$-invariant subbundles are actually reducing subbundles for $E$ (see Proposition 4.18, Chapter 1 in \cite{Koba}). Therefore we can apply
the above theorem  to obtain a collection of mutually orthogonal reducing subbundles $\{E_{ij} \}, 1\leq i\leq r, 1\leq j\leq m_i$ with $rank E_{ij}=n_i$ , such that
\[ E=E_{11}\oplus \cdots \oplus E_{1 m_1}\oplus \cdots \oplus E_{r1}\oplus\cdots \oplus E_{rm_r}. \]
Further, if we apply the theorem to the bundle $E(T)$ with canonical connection for a Cowen-Douglas operator $T$ , then $\{E_{ij}\}$ correspond to reducing subspaces $\{ \mathcal{H}_{ij} \}:=\bigvee_{z\in \Omega_0} E_{ij}(z) $ such that
\[ \mathcal{H}=\mathcal{H}_{11}\oplus \cdots \oplus \mathcal{H}_{1 m_1}\oplus \cdots \oplus \mathcal{H}_{r1}\oplus\cdots \oplus \mathcal{H}_{rm_r}. \]
\end{remark}

\begin{theorem}[\cite{ChD}, Theorem 4.4]\label{reducing-subspace}

 Let $T$ be a Cowen-Douglas operator and $E(T)$ be its associated holomorphic
Hermitian vector bundle with reducing subbundles $\{E_{ij}\}$ given by the block diagonalization of the canonical connection. Let $\{ \mathcal{H}_{ij} \}$ be the corresponding reducing subspaces. Then:
\begin{enumerate}
  \item The reducing subspaces $\{ \mathcal{H}_{ij} \}$ are minimal.
  \item $\{ \mathcal{H}_{ij} \}$ and $\{ \mathcal{H}_{i' j'} \}$ are unitarily equivalent if and only if $i = i'$ .
\end{enumerate}

\end{theorem}
Moreover, we can give the geometrical characterization of the von Neumann algebra $V^*(T):=\{T,T^*\}'$:
\begin{theorem}[\cite{ChD}, Theorem 4.5]\label{von-alg}

For a Cowen-Douglas operator $T$ in $\mathfrak{B}_n(\Omega)$, the von Neumann algebra $V^*(T)$ is isomorphic to the commutant of the matrix algebra $M(\textbf{n}, \otimes \textbf{m})$ in $M_n(\mathbb{C})$, where $M(\textbf{n}, \otimes \textbf{m})$ is given by the block diagonalization of the canonical connection on $E_T$. In particular, $V^*(T)$ is abelian if and only if $m_i=1$ for all $i$, i.e., there is no multiplicity in the block diagonalization of the canonical connection on $E_T$.
\end{theorem}

Back to the topic,  we will characterize when $S_z^*$ on $\mathcal{K}_\theta$ induced by a rational inner function is a Cowen-Douglas operator.


For the general rational inner functions, the following example shows that $S_z^*$ fails to be in the Cowen-Douglas class.
\begin{example}\label{counterexample}
Take $0<c<1/8$, define
\[ g(z,w)=\frac{2z^2w-z+c}{2-zw+cz^2w}.\]
Then  $g\in A(\mathbb{D}^2)$ is an irreducible inner function, that  is $g$ is not a product of two nonconstant inner functions.   However the  zero set  $E=Z_g \cap \mathbb{D}^2$ of $g$ is disconnected: $E$ contain $(c,0)$ and $(1/2, 1-2c)$, but no point with $|z|=1/4$, for if $|z|=1/4$ and $|w|<1$, then $|2z^2 w|<1/8$ whereas $|z-c|>1/8$.
\end{example}
There do exist two inner functions $g_1, g_2$ on $\mathbb{D}^2$ such that $Z_{g_i}=E_i$, where $E_1,E_2$ are the components of $E$. To be precisely,  we factorize $2z^2w-z+c$ as
\[2z^2w-z+c=2w\left(z-\frac{1-\sqrt{1-8cw}}{4w}\right)\left(z-\frac{1+\sqrt{1-8cw}}{4w}\right),\]
here we take the negative real axis as the branch line for the function $\sqrt{z}$, i.e., if we write $z=re^{i\theta}, -\pi<\theta\leq\pi$, then $\sqrt{z}=\sqrt{r} e^{i\theta/2}$. In this case, for $|w|<1$, then $$\sqrt{1-8cw}=x+iy,x>0.$$
Therefore
\[|1-\sqrt{1-8cw}|<|1+\sqrt{1-8cw}|\]and\[|1+\sqrt{1-8cw}|>1,\]
 and it implies $|\frac{|1+\sqrt{1-8cw}|}{4w}|>1/4$. Note that
\[|1-\sqrt{1-8cw}||1+\sqrt{1-8cw}|=8c|w|,\]
hence $|\frac{|1-\sqrt{1-8cw}|}{4w}|<2c<1/4$.
Define
\begin{equation*}
h_1(z,w)=
\begin{cases}
  z-\frac{1-\sqrt{1-8cw}}{4w} &if\quad w\neq 0\\
  z-c&if\quad w=0
\end{cases}
\end{equation*}
\[ h_2(z,w)=2zw-(1+\sqrt{1-8cw})/2 .\]
As $1-\sqrt{1-8cw}\sim 4cw$ when $w\rightarrow 0$ , then we  see that $h_1,h_2\in A(\mathbb{D}^2)$. Let $E_1=Z_{h_1} \cap \mathbb{D}^2$, $E_2=Z_{h_2}\cap \mathbb{D}^2$, then $E_1, E_2$ are exactly the two connected component of $E$.
The following lemma implies the existence of $g_i$.
\begin{lemma}[\cite{Rud},    Theorem 5.4.5]\label{zeros}
Suppose $f\in H^\infty(\mathbb{D}^n)$, $f \nequiv 0$, $\psi$ is upper semi-continuous on $\mathbb{T}^n$, and $|f^*|=\psi $ a.e., where $f^*$ means the boundary value of $f$. Then $f$ has the same zeros as some inner function $g$. In particular, every nonzero $f\in A(\mathbb{D}^n)$ has the same zeros as some inner function $g$.
\end{lemma}

It is worth noting that $g/g_i\notin H^2(\mathbb{D}^2)$ (recall that $g$ is an irreducible inner function), although it is holomorphic in $\mathbb{D}^2$ and $|g/g_i|=1 $ a.e. on $\mathbb{T}^2$.

We claim that for $\mathcal{K}_g=H^2(\mathbb{D}^2)\ominus g H^2(\mathbb{D}^2)$, $S_z^*$ does not belong to the Cowen-Douglas class. If $S_z^*\in \mathfrak{B}_1(\Omega^*)$, then either $\Omega=\pi_1(E_1)\subseteq \{z\in \mathbb{D}: |z|<1/4\}$ or $\Omega=\pi_1(E_2)\subseteq \{z\in \mathbb{D}: |z|>1/4\} $. Note that $S_z^*\in \mathfrak{B}_1(\Omega^*)$  if and only if
\[
  \bigvee_{\lambda\in \Omega} KerS_{z-\lambda}^* =\bigvee_{\lambda \in \Omega}  \left\{\frac{1}{1-\overline{\lambda}z} \frac{1}{1-\overline{\eta}w} : g(\lambda, \eta)=0\right\}=\mathcal{K}_g.\]
 If $\Omega=\pi_1(E_1)$, let $P_N$ be the orthogonal projection onto   $\mathcal{K}_g$ and consider the function $P_N g_1$. It is  obvious that $P_N g_1$ is not zero since $g_1\notin  gH^2(\mathbb{D}^2)$. For $\lambda\in \pi_1(E_1) $, $g(\lambda,\eta)=0$ if and only if $g_1(\lambda,\eta)=0$, hence
 \[ \langle P_N g_1, \frac{1}{1-\overline{\lambda}z} \frac{1}{1-\overline{\eta}w} \rangle =0.\]
 So $P_N g_1 \perp \bigvee_{\lambda\in \Omega} Ker~S_{z-\lambda}^* $, which is a contradiction. For the case of $\Omega=\pi_1(E_2)$, a similar discussion for $P_N g_2$ holds.

 By \cite{ZYL}, for the $(n,1)$-type rational inner function which has no one-variable inner factors, $S_z$ is  irreducible, hence for the above counterexample, $S_z^*$ is still irreducible,  although it is not in $\mathfrak{B}_1$. We will take a closer look at this class of compressed shift operators in the next section.

The above counterexample shows that the connectness of the projection of the zero set is crucial. By Theorem \ref{spec2}, we know that
 $$\sigma_e(S_z)=\mathbb{T}\cup \pi_1(Z_\theta \cap \mathbb{D}\times \mathbb{T}).$$

\begin{lemma}\label{lemma3.9}
$\sigma_e(S_z)$ consists of the unit circle $\mathbb{T}$ and finitely many curves in $\mathbb{D}$.  The unit disk is divided into finitely many connected components by these finite curves and $\mathbb{T}$.
\end{lemma}
\begin{proof}
Write $\theta=\widetilde{p}/p$, by Theorem 12.1 in \cite{Bli}, the set $\{~z\in \mathbb{C}: \widetilde{p}(z,e^{it})=0~for~some~t\in [0,2\pi)~\}$ consists of finitely many distinct continuous functions $z_1(t), z_2(t),\cdots, z_m(t)$. Note that when $z_i(t)$ exits the unit disk, $(z_i(t), e^{it} )$ exits the closed bidisk from $\mathbb{T}^2$. Since $\theta$ has at most finitely many zeros on $\mathbb{T}^2$, we get that every $z_i(t)$ enters or exits the unit disk finitely many times. It is clear that
$$\pi_1(Z_\theta \cap \mathbb{D}\times \mathbb{T})=\bigcup_{i} \left( \{~z_i(t)~\}\bigcap \mathbb{D}\right),$$
thus $\sigma_e(S_z)$ consists of the unit circle $\mathbb{T}$ and finitely many curves in $\mathbb{D}$. Moreover, if $\{~z_i(t)~\}\bigcap \mathbb{D}$ is connected, then either it is entirely contained in $\mathbb{D}$ or its closure has only one intersection with $\mathbb{T}$. If $\{~z_i(t)~\}\bigcap \mathbb{D}$ is not connected, then it consists of finitely many connected curves which intersect with $\mathbb{T}$.

Obviously, the unit disk is divided into some connected components by these finite curves and $\mathbb{T}$. For each connected component $\Omega_i$, it is a Fredholm domain for $S_z^*$.
Since $Ker S_z-\lambda$ is always $\{0\}$ and the Fredholm index is an integer valued continuous function, hence the number of zeros of $\theta(\lambda,\cdot)$ is constant for $\lambda\in \Omega_i$.

To show there exists at most finitely many connected components, it is suffice to prove that for $i\neq j$, the set $\{z_i(t)\}\bigcap\mathbb{D}$ and the set $\{z_j(t)\}\bigcap\mathbb{D}$ have at most finitely many intersections. Assume not this, then there two infinite sequence $\{(z_n,w_n)\}$ and $\{(z_n,w_n')\}$ with $|w_n|=|w_n'|=1$ and $z_n\in \{z_i(t)\}\bigcap\{z_j(t)\}\bigcap\mathbb{D}$ such that $\widetilde{p}(z_n,w_n)=\widetilde{p}(z_n,w_n')=0$. Connect the points $\{z_n\}$ with a continuous arc, namely $z(t)$,  again by Theorem 12.1 in \cite{Bli}, there is finitely many distinct continuous functions $\{ w_k(t)\}$ such that $\widetilde{p}(z(t),w_k(t))=0$. Let $z_n=z(t_n)$, then $w_n=w_{k_n}(t_n)$ for some $k_n$, since the range of $k_n$ is a finite set, there is a subsequence of $\{z_n\}$ (for simple, still write as $\{z_n\}$) such that $w_n=w_n'=w_k(t_n)$ for some $k$. It implies that $z_i(t)=z_j(t)$ has infinitely many intersections in $\mathbb{D}$, but by the implicit function theorem, each $z_i(t)$ is locally holomorphic, this is contradict to the fact $i\neq j$.

\end{proof}

As we have shown in the proof of Lemma \ref{lemma3.9}, the intersection of the finitely many curves and $\mathbb{T}$ are exactly the set $\pi_1(Z_\theta\cap \mathbb{T}^2)$, which is  at most a finite set  $\{ z_0,z_1,\cdots, z_n \}$ (arrange $z_i$ counterclockwise). If $n\geq 1$, let $\Omega_i$ be corresponding component whose boundary contains the arc $\widehat{z_iz_{i+1}}$, where the arc $\widehat{z_nz_{n+1}}$ means $\widehat{z_n z_0}$. If $Z_\theta\cap \mathbb{T}^2$ is an empty set or a  single point set, then let $\Omega$ be corresponding component whose boundary contains the whole unit circle.  We then have the following lemma.
\begin{lemma}\label{existence}
Suppose $\theta=\prod\limits_{j} \theta_j^{k_j}$ has no inner factors depending only on $z$, where $k_j>0$ and $\theta_j=\widetilde{p_j}/p_j $ is irreducible. Then
\[\Omega_i\subseteq  \pi_1(Z_{\theta_j}\cap \mathbb{D}^2 )\]
for every $i, j$.
\end{lemma}
\begin{proof}
Let \[\widetilde{p_j}(z,w)=\sum\limits_{n\geq 0} a_n(z)w^n.\]
Then there are only finitely many $z\in \mathbb{T}$, denoted by $E=\{\zeta_1,\ldots,\zeta_k\},$ such that  $\widetilde{p_j}(z,w)$ has no zeros in $\mathbb{C}$, i.e.,
\[\widetilde{p_j}(z,w)=a_0(z), \forall w\in \mathbb{C},\]
otherwise, there are   infinitely many $z\in \mathbb{T}$ such that
\[\widetilde{p_j}(z,w)=a_0(z)\]
holds for all $w\in \mathbb{C}$,  which means that  for  $\forall w\in \mathbb{C}$,  the polynomial $\widetilde{p_j}(z,w)-a_0(z)$ has infinitely many zeros in $\mathbb{T}$,  and we  obtain that
\[\widetilde{p_j}(z,w)\equiv a_0(z),\]
and this is a contradiction.  Therefore for any $e^{it}\in \widehat{z_iz_{i+1}}\setminus E$,  $\widetilde{p_j}(e^{it}, w)$ does exist zeros in $\mathbb{C}$. Since $\widetilde{p_j}$  has no zeros in $\mathbb{T}\times \mathbb{E}$, where $\mathbb{E}=\{z: |z|>1\}$, thus the zeros of  $\widetilde{p_j}(e^{it}, w)$ must all lie in $\mathbb{D}$. It follows that for sufficiently small $r>0$, whenever $z$ varies in the ball $B(e^{it}, r)$ centered at $e^{it}$ with radius $r$, the zeros of $\widetilde{p_j}(z,w)$   also lie in $\mathbb{D}$, which gives that
\[B(e^{it},r)\cap \mathbb{D} \subset \pi_1(Z_{\theta_j} \cap \mathbb{D}^2).\]
Note that
\[ \pi_1(Z_{\theta_j} \cap \mathbb{D}\times \mathbb{T})\subseteq \pi_1(Z_{\theta} \cap \mathbb{D}\times \mathbb{T}),  \]
so every  connected component of $\mathbb{D}$ partitioned by  $\pi_1(Z_{\theta} \cap \mathbb{D}\times \mathbb{T})\cup \mathbb{T}$ is contained in some  connected component of $\mathbb{D}$ partitioned by $\pi_1(Z_{\theta_j} \cap \mathbb{D}\times \mathbb{T})\cup \mathbb{T}$. Therefore
$$\Omega_i\subseteq  \pi_1(Z_{\theta_j}\cap \mathbb{D}^2).$$
\end{proof}
By Lemma \ref{existence}, for a rational inner function $\theta$ without inner factors depending only on $z$, each $\Omega_i \times \mathbb{D}$ has nonempty intersection with every algebraic component of the zero set of $Z_\theta$.

\begin{theorem}\label{3.10}
Suppose  $\theta=\prod\limits_{j} \theta_j^{k_j}$ is a rational inner function without   inner factor depending only on  $z$, where $k_j>0$ and $\theta_j=\widetilde{p_j}/p_j$ is irreducible inner factor.  If the projection of the zero set of  every  $\theta_j$ onto the $z$-plane   is connected in $\mathbb{D}$, i.e., $\pi_1(Z_{\theta_j}\cap \mathbb{D}^2)$ is connected, then $S_z^*$ is a Cowen-Douglas operator for some domain. Moreover, the connectness condition is also necessary.
\end{theorem}

\begin{proof}

Suppose that each  $\pi_1(Z_{\theta_j}\cap \mathbb{D}^2)$ is connected. Take $\Omega$ be an arbitrary $\Omega_i$, then $\Omega \times \mathbb{D}$ has nonempty intersection with  each $Z_{\theta_j}$.
To show $S_z^*$ is Cowen-Douglas on $\Omega^*$, it suffices to show that
\[ \bigvee_{\lambda\in \Omega} Ker S_{z-\lambda}^* =\mathcal{K}_{\theta}. \]
Suppose   $f\in \mathcal{K}_{\theta}$   is orthogonal to
 $$\bigvee\limits_{\lambda\in \Omega} Ker S_{z-\lambda}^*=\left\{\frac{g(w)}{1-\overline{\lambda}z} :~ g(w)\in H_w^2\ominus \theta(\lambda,w)H_w^2\right\},$$
then  for each $\theta_j$, $f/\theta_j$ is holomorphic on $\Omega\times \mathbb{D}$. In fact, as $\theta_j$ is irreducible, there are at most finitely many points  $(\lambda,w)$ in $\mathbb{D}^2$ such that $\theta_j(\lambda, w)=\frac{\partial \theta_j(\lambda,w)}{\partial \lambda}=0$ and hence  $\theta_j(\cdot, w)$ has only simple zeros except for finitely many $w$. Therefore except for finitely many $w$, $f(\cdot,w)/\theta_j(\cdot,w)$ is holomorphic in $\Omega$, and obviously, $f(\lambda,\cdot)/ \theta_j(\lambda,\cdot)$ is holomorphic in $\mathbb{D}$, so by Hartogs's Theorem, $f/\theta_j$ is holomorphic on $\Omega\times \mathbb{D}$. So for every $\theta_j$, $f=0$ on $Z_{\theta_j} \cap \Omega\times \mathbb{D}$. By the continuations of the values of an algebra function (cf.\cite{Bli}, Chapter \textrm{2}), that is,
\[ f(z,\rho_i(z))=0~~for~~all~~z\in \Omega~~and~~i=1,\cdots,N_j,\]
where $\{\rho_i(z)\}$  are the analytic branches of the algebra function defined by $\theta_j(z,w)=0$.
 Since $\pi_1(Z_{\theta_j}\cap \mathbb{D}^2)(\supset \Omega)$ is connected, by the uniqueness of analytic continuation, we have $f=0$ on $\pi_1(Z_{\theta_j} \cap  \mathbb{D}^2)$, and hence  $f/\theta_j$ is holomorphic on $\mathbb{D}^2$. Repeating the above discussion  on $f/\theta_j$, we get that  $f/\theta_j^{k_j}$ is holomorphic on $\mathbb{D}^2$. It is follows that   $f/\theta$ is holomorphic on $\mathbb{D}^2$.  Since $\theta$ is a good inner function, we get $f\in \theta H^2(\mathbb{D}^2)$, and this gives that    $f$ must be zero.  This completes the proof of the first part of the theorem.

 Now suppose  $\pi_1(Z_{\theta_j})$ is not connected in $\mathbb{D}$ for some irreducible factor $\theta_j$.  As the projection map is continuous, hence $Z_{\theta_j}$ has  disjoint analytic components in $\mathbb{D}^2$. This means that the numerator of $\theta_j$, say $q_j$,  has disjoint analytic components in $\mathbb{D}^2$, and if we  partition $\pi_1(Z_{\theta_j})$  into disjoint components $\Omega_i$, then for every component $\Omega_i$ , there must exist  finite analytic components (take as many as possible) of $q_j$  such that $\Omega_i$ is exactly the union of the projection of those analytic components.
Therefore by Lemma \ref{zeros}, for every component $\Omega_i$ , there exists an inner function $g_i$ such that  $\pi_1(Z_{g_i})=\Omega_i$ and $g_i$ has the same zeros as $\theta_j$ on $\Omega_i\times \mathbb{D}$ .  Note that $\theta g_i/\theta_j\notin \theta H^2(\mathbb{D}^2)$ and
\[ \frac{\theta g_i}{\theta_j}\perp  \bigvee_{\lambda \in \Omega_i}  \left\{\frac{1}{1-\overline{\lambda}z} \frac{1}{1-\overline{\eta}w} : \theta(\lambda, \eta)=0\right\}.  \]
 It follows that
\[
  \bigvee_{\lambda\in \Omega_i} Ker~S_{z-\lambda}^* =\bigvee_{\lambda \in \Omega_i}  \{\frac{1}{1-\overline{\lambda}z} \frac{1}{1-\overline{\eta}w} : \theta(\lambda, \eta)=0\}\neq \mathcal{K}_{\theta},\]
and therefore $S_z^*$ is not a Cowen-Douglas operator on $\Omega_i^*$.  This completes the proof.
\end{proof}

\section{The Generalized Cowen-Douglas Operator}
In this section, we     study the elementary   geometric  properties for     the generalized Cowen-Douglas operators, and prove that $S_z^*$ on $\mathcal{K}_{\theta}$ induced by a rational inner function is  a generalized Cowen-Douglas operator  and  $S_z$ is reducible if and only if $E_{S_z^*}$ is strictly   reducible.
Recall that for an operator $T$  in $\mathfrak{B}_{\alpha}(\Omega)$ and a connected component $\Omega_i$ of $\Omega$,   $(E_T(\Omega_i), \pi_i)$  is Hermitian holomorphic vector bundle over $\Omega_i$ of dimension $\alpha_i$ and
\[ E_T= \bigsqcup_{i=1}^N E_T(\Omega_i). \]
\begin{definition}
We say that $E_T$ is  reducible  if there exist nontrivial (not equals $\{0\}$ or $E_T$) $E^1$  and $E^2$ with
\[ E^1= \bigsqcup_{i=1}^N E^1(\Omega_i),~~~~E^2= \bigsqcup_{i=1}^N E^2(\Omega_i)\]
such that
\begin{enumerate}
  \item (Reducibility condition) $E^1(\Omega_i)\oplus E^2(\Omega_i)=E_T(\Omega_i)$ for all $i=1,\cdots,N$.
  \item (Orthogonality condition) $E^1(\Omega_i)\perp E^2(\Omega_j)$ for any $i,j=1,\cdots,N$, that is,
      \[ \bigvee_{w\in \Omega_i} E^1(w) \perp \bigvee_{w\in \Omega_j} E^2(w).\]
\end{enumerate}
In this case, we write as
\[E_T=E^1\oplus E^2.\]
\end{definition}
It is worth noting that $E^1(\Omega_i)$ may be zero   bundle for some $i$ since we only require $E^1$ is nontrivial. The following example provides a generalized Cowen-Douglas operator $T$ such that $E_T$ is reducible, however, each $E_T(\Omega_i)$ is irreducible as a Hermitian holomorphic  bundle.

\begin{example}\label{key-example}
Let $\mathcal{H}=H^2(\mathbb{D})\oplus H^2(\mathbb{D})$ and $M_z$ be the shift on $H^2(\mathbb{D})$. Let $T=M_z^*\oplus M_{z+2}^*$,  $\Omega_1=\mathbb{D}$, $\Omega_2=2+\mathbb{D} $,  $\Omega=\mathbb{D}\bigcup 2+\mathbb{D}=\bigcup\limits_{i=1}^2 \Omega_i$, then
\begin{equation*}
Ker(T-\bar{w} )=
\begin{cases}
 \mathbb{C}\left(
             \begin{array}{c}
               K_w \\
               0 \\
             \end{array}
           \right)
  &if\quad w\in  \Omega_1\\

   \mathbb{C}\left(
             \begin{array}{c}
               0 \\
               K_{w-2} \\
             \end{array}
           \right)
  &if\quad w\in  \Omega_2.
\end{cases}
\end{equation*}
It is clear that $T\in \mathfrak{B}_{\alpha}(\Omega^*)$ with $\alpha=(1,1)$.

Take
\begin{align*}
  E^1(\Omega_1)&=\left\{~\mathbb{C}\left(
             \begin{array}{c}
               K_w \\
               0 \\
             \end{array}
           \right) ~\right\} & E^1(\Omega_2)&=\{0\}, \\
  E^2(\Omega_1)&=\{0\} &E^2(\Omega_2)&=\left\{~\mathbb{C}\left(
             \begin{array}{c}
               0 \\
               K_{w-2} \\
             \end{array}
           \right)~ \right\}.
\end{align*}
It is clear that $E_T=E^1\bigoplus E^2$, and hence it is reducible.
\end{example}


In the following, we will give some elementary   geometric  properties for     the generalized Cowen-Douglas operators.
\vspace{0.5cm}

We say $w_0$ is  \textbf{a point of stability}  for $T$ if $T-w_0$ is Fredholm and $\dim Ker(T-w)$ is constant on some neighborhood of $w_0$. Obviously, for an operator $T$ in $\mathfrak{B}_{\alpha}(\Omega)$, any point $w$ in $\Omega$ is a point of stability for $T$.

\begin{lemma}[\cite{CoD}, Proposition 1.11] \label{basis}

If $w_0$ is a point of stability for $T$ in $\mathcal{L}(\mathcal{H})$, then there exist holomorphic $\mathcal{H}$-valued functions $\{e_i(w)\}_{i=1}^n$ defined on some neighborhood $\Delta$ of $w_0$ such that $\{e_1(w), e_2(w),\cdots,e_n(w)\}$ forms a basis of  $Ker(T-w)$ for $w$ in $\Delta$.
\end{lemma}

By  Lemma \ref{basis}, we have the following proposition.
\begin{proposition}
For  an operator $T$ in $\mathfrak{B}_{\alpha}(\Omega)$ and a connected component $\Omega_i$ of $\Omega$, let $\Delta$ be a connected open subset of $\Omega_i$, then
\[ \bigvee_{w\in \Delta} Ker(T-w)=\bigvee_{w\in \Omega_i} Ker(T-w) .\]
Moreover, by the spanning property (cf.\cite{CoD},section 1.7), for any $w_0\in \Omega_i$, we have
\[ \bigvee_{k=1}^\infty Ker(T-w_0)^k  =\bigvee_{w\in \Omega_i} Ker(T-w).\]
\end{proposition}
\begin{proof}
Suppose $x$ in $\mathcal{H}$ is orthogonal to $Ker(T-w)$ for $w$ in $\Delta$. If $w_0$ is a boundary point of $\Lambda=$ interior of $\{ w\in \Omega : x\perp Ker(T-w) \}$ in $\Omega$ (obviously, $\Delta\subset \Lambda$, hence $\Lambda$ is nonempty) , then by Lemma \ref{basis}, there exist an open set $\Lambda_0$ of $\Omega$ about $w_0$ and holomorphic functions $\{e_1(w), e_2(w),\cdots,e_n(w)\}$ defined on $\Lambda_0$ which form a basis for $Ker(T-w)$ for each $w$ in $\Lambda_0$. Since the holomorphic functions $\langle e_i(w),x \rangle$ for $i=1,2,\cdots,n$ vanish on $\Lambda$, they  vanish identically and hence $\Lambda_0$ is contained in $\Lambda$. Thus $\Lambda=\Omega$ which completes the proof.
\end{proof}

There are many properties of Cowen-Douglas operators can be extend to the generalized Cowen-Douglas operators and the proof is similar to the original proof. For reader's convenience, we  give the original references and the similar proofs.

\begin{proposition}[\cite{CoD}, Proposition 1.21]\label{commutant-CD}
For $T$ being in the generalized Cowen-Douglas class $\mathfrak{B}_{\alpha}(\Omega)$, let $H^\infty_{\mathcal{L}(E_T)}(\Omega)$ be the collection of bounded endomorphisms on $E_T$.  Then there is a contractive monomorphism $\Gamma_T$ from the  commutant $\{T\}'$ of $T$ into $H^\infty_{\mathcal{L}(E_T)}(\Omega)$.
\end{proposition}
\begin{proof}
If $XT=TX$, then $X Ker (T-w) \subseteq Ker (T-w)$ for $w$ in $\Omega$. Moreover, if $e(w)$ is a local holomorphic cross-section of $E_T$, then so is $X e(w)$. Therefore $X$ defines a holomorphic bundle map $\Gamma_T X$ on $E_T$ by $(\Gamma_T X)(w)=X|_{Ker(T-w)}$. Since
\[\|(\Gamma_T X)(w)\|=\|X|_{Ker(T-w)}\|\leq \|X\| , \]
we see that $\Gamma_T X$ lies in $H^\infty_{\mathcal{L}(E_T)}(\Omega)$ and $\Gamma_T$ is contractive. That $\Gamma_T$ is a homomorphism is clear and it is one-to-one since $\bigvee_{w\in \Omega} Ker(T-w)=\mathcal{H}$.
\end{proof}

In general, $\Gamma_T$ is not onto. For example,  when $T$ is the Dirichlet operator (cf. \cite{Taylor}), $\Gamma_T$ is not onto. Further, similar to the proof of Corollary 3.7 in \cite{ChD}, we can show that for a generalized Cowen-Douglas operator $T$, if $\Omega$ has only finitely many connected components, then its von Neumann algebra $V^*(T)$ is finite dimensional.
\begin{proposition}[\cite{ChD}, Corollary 3.7]
For a generalized Cowen-Douglas operator $T$ over $\Omega$, assume $\Omega$ has only finitely many connected components, then its von Neumann algebra $V^*(T)$ is finite dimensional.
\end{proposition}
\begin{proof}
Assume
\[\Omega=\bigcup_{i=1}^N \Omega_i,~~~N<\infty.\]
For $w_i\in \Omega_i$, since $Ker(T-w_i)$ is finite dimensional, it   suffices  to show that for a operator $S\in V^*(T)$, $S$ is completely determined by its action on $\bigvee_{i=1}^N Ker(T-w_i)$. That is, if $S\in V^*(T)$ and $S|_{\bigvee_{i=1}^N Ker(T-w_i)}=0$, then $S=0$.

For any $i$, since   $Ran(T-w_i)^*$ is closed and $\mathcal{H}=Ker(T-w_i)\oplus Ran (T-w_i)^*$, we have
\begin{align*}
   S\mathcal{H} &=S(T-w_i)^* \mathcal{H}=(T-w_i)^* S\mathcal{H}=(T-w_i)^* S(T-w_i)^*\mathcal{H}  \\
   & =[(T-w_i)^*]^2 S\mathcal{H}=\cdots \subset \bigcap_{k=1}^\infty Ran[(T-w_i)^*]^k.
\end{align*}
Hence
\[ S\mathcal{H}\subset \bigcap_{i=1}^N \bigcap_{k=1}^\infty Ran[(T-w_i)^*]^k=\left(\bigvee_{i=1}^N \bigvee_{i=1}^\infty Ker(T-w_i)^k \right)^\perp.\]
By the spanning property and $\bigvee_{i=1}^N \bigvee_{w\in \Omega_i} Ker(T-w)=\mathcal{H}$, we have $S\mathcal{H}=0$, and this completes the proof.
\end{proof}


\begin{theorem}[\cite{CoD}, Propositon 1.18]\label{reduce}
An operator $T$ in $\mathfrak{B}_{\alpha}(\Omega)$ is reducible if and only if $E_T$ is reducible.
\end{theorem}
\begin{proof}
Suppose $\mathcal{H}=\mathcal{M}\oplus \mathcal{N}$, where $\mathcal{M}$ and $\mathcal{N}$ reduce $T$. Let  $x$ be in $Ker (T-w)$,   write $x=x_1\oplus x_2$ for $x_1\in \mathcal{M}$ and $x_2$ in $\mathcal{N}$.  Then $$Tx_1\oplus Tx_2=wx=wx_1\oplus wx_2,$$
and this implies  that both $x_1$ and $x_2$ are in $Ker(T-w)$. It follows that $Ker(T-w)$ decomposes into
\[\left\{Ker(T-w)\cap \mathcal{M}\right\} \oplus \left\{Ker(T-w)\cap \mathcal{N}\right\}, \]
and hence $T_1=T|_{\mathcal{M}}$ is in $\mathfrak{B}_{\beta}(\Omega)$ and $T_2=T|_{\mathcal{N}}$ is in $\mathfrak{B}_{\gamma}(\Omega)$, where $\alpha=\beta+\gamma$, i.e., $\alpha_i=\beta_i+\gamma_i$ for  $i=1,2$. Therefore $E_T=E_{T_1}\oplus E_{T_2}$ is reducible.

Now suppose $E_T=E^1\oplus E^2$ is a reduction of $E_T$. For every $i$, fixed $w_i$ in $\Omega_i$ and let $\Delta_i$ be a neighborhood of $w_i$ in $\Omega_i$ on which there is a trivilization $e_1(w),\cdots,e_{\alpha_i}(w)$ of $E_T(\Omega_i)$, where $e_1(w),\cdots,e_{\beta_i}(w)$ span $E^1(\Omega_i)$ and $e_{\beta_i+1}(w),\cdots, e_{\alpha_i}(w)$ span $E^2(\Omega_i)$. To show that $T$ is reducible, it is enough to show that for every $i$, $ \bigvee_{w\in \Delta_i} \{~e_1(w),\cdots,e_{\beta_i}(w)~\}  $  and $  \bigvee_{w\in \Delta_i} \{~e_{\beta_i+1}(w),\cdots, e_{\alpha_i}(w)~\}  $ are orthogonal.  As $\langle e_s(w), e_t(w) \rangle=0$ for $1\leq s\leq \beta_i<t\leq \alpha_i$ and differentiating with respect to $w$ (viewing the functions of the variables $w$ and $\bar{w}$) we obtain
\[ 0=\frac{\partial}{\partial w} \langle e_s(w), e_t(w) \rangle=\langle e_s'(w), e_t(w) \rangle .  \]
Similarly, we have $\langle e_s^{(k)}(w), e_t(w) \rangle=0$ for $k=0,1,2,\cdots$ and therefore
\[ \langle~ e_s(w), e_t(w_i) ~\rangle=\langle~ \sum_{k=0}^\infty \frac{e_s^{(k)}(w_i)}{k!}(w-w_i)^k, e_t(w_i) ~\rangle =0 \]
for $|w-w_i|$ sufficiently small. Recall   that $\bigvee_{w\in \Delta_i} Ker (T-w)=\bigvee_{w\in \Omega_i} Ker (T-w)$, thus we proved that $e_t(w_i)$ is orthogonal to $ \bigvee_{w\in \Delta_i} \{~e_1(w),\cdots,e_{\beta_i}(w)~\}  $. Since $w_i$ is arbitrary and again by the fact $\bigvee_{w\in \Delta_i} Ker (T-w)=\bigvee_{w\in \Omega_i} Ker (T-w)$,   we finish the proof.
\end{proof}


The following theorem is main theorem of this section, which shows that $S_z^*$ is a generalized Cowen-Douglas operator.
\begin{theorem}\label{4.7}
Suppose $\theta(z,w)$ is a rational inner function  without   inner factors depending only on variable $z$ and   $\mathcal{K}_{\theta}=H^2(\mathbb{D}^2)\ominus \theta H^2(\mathbb{D}^2)$ is the Beurling type quotient module.  Let
 $$\Omega=\sigma(S_z)\setminus \sigma_e(S_z)=\bigcup_i \Omega_i,$$
 where $\Omega_i, i=1,\ldots, N$ are the connected components of $\Omega$,
 then $S_z^*$ belongs to $\mathfrak{B}_{\alpha}(\Omega^*)$, where $\alpha_i=Z_\theta|_{\Omega_i}$ is the numbers of zeros of $\theta(\lambda, \cdot)$ in $\mathbb{D}$, $\lambda\in \Omega_i$. 
\end{theorem}
\begin{proof}
By Theorem \ref{kernel} and Corollary \ref{spec2}, we see that for $\lambda\in \Omega_i$, $Ran(S_{z-\lambda}^*)=\mathcal{K}_{\theta}$  and the dimension of $kerS_{z-\lambda}^*$ is equal to $n_i$, which is a constant.
 To see  that $S_z^*\in \mathfrak{B}_{\mathbf{n}}(\Omega^*)$, suppose $f\in \mathcal{K}_{\theta}$ and $f$ is orthogonal to
$$\bigvee\limits_{\lambda\in \Omega} Ker S_{z-\lambda}^*=\left\{\frac{g(w)}{1-\overline{\lambda}z} :~ g(w)\in H_w^2\ominus \theta(\lambda,w)H_w^2\right\},$$
then by a similar argument with the proof of Theorem \ref{3.10}, we can obtain that $f/\theta$ is holomorphic in $\mathbb{D}^2$. As $\theta$ is a  good inner function, by Lemma \ref{Rudin}, we know that $f\in \mathcal{K}_{\theta}$, and  hence $f=0$.  Therefore
\[ \bigvee\limits_{\lambda\in \Omega} Ker S_{z-\lambda}^*=\mathcal{K}_{\theta}, \]
and this gives that $S_z^*\in \mathfrak{B}_{\mathbf{n}}(\Omega^*)$. The proof is completed.
\end{proof}

By Theorem \ref{reduce} and \ref{4.7},   $S_z$ is reducible if and only if $E_{S_z^*}$ is reducible. As seen in Example \ref{key-example}, there exists a generalized Cowen-Douglas operator $T$ such that each $E_T(\Omega_i)$ is irreducible as a   Hermitian holomorphic  bundle, and $T$ is reducible. However, in the case of $\theta$ has no univariate inner factors,  we will see later that if $E_{S_z^*}(\Omega_i)$ is irreducible for some $i$, then $S_z$ is irreducible.
\begin{definition}\label{Def4.9}
Let $T$ be a generalized Cowen-Douglas operator. We say  $E_T$ is    strictly reducible   if $E_T$ is reducible and for each reducible decomposition  $E_T=E^1\bigoplus E^2$,  $E^1(\Omega_i)$ is nontrivial for any $i$, that is, $E^1(\Omega_i)$ is not $\{0\}$ or $E_T(\Omega_i)$.
In particular, if $E_T$ is    strictly reducible, then each  $E_T(\Omega_i)$ is  also  reducible.
\end{definition}

\begin{theorem}\label{strictly-reducible}
Suppose $\theta(z,w)$ is a rational inner function  without  univariate inner factors and   $\mathcal{K}_{\theta}=H^2(\mathbb{D}^2)\ominus \theta H^2(\mathbb{D}^2)$ is the associated Beurling type quotient module. Then  $S_z$ is reducible if and only if $E_{S_z^*}$ is strictly reducible.
\end{theorem}

\begin{proof} It suffices to show that if $E_{S_z^*}$ is  not strictly reducible, then $S_z$ is irreducible.
If $\Omega$ is connected,  by Lemma \ref{reducibility},  the conclusion is clear. Now assume that $\Omega=\bigcup\limits_{i=1}^N \Omega_i$ has at least two connected components, where $\Omega_i, i=1,2\ldots, N$ are the connected components of $\Omega$, $N\geq 2$. For convenience, we deal with the complex conjugate of $E_{S_z^*}$, and write
\[ E_{S_z^*}= \bigsqcup_{i=1}^N E_{S_z^*}(\Omega_i), \]
where $E_{S_z^*}(\Omega_i)=\left\{(\lambda,x)\in \Omega_i\times \mathcal{K}_{\theta}:    x\in Ker~ S_{z-\lambda}^*\right\}$ is a Hermitian anti-holomorphic bundle over $\Omega_i$.  Next we will show that $E_{S_z^*}$ is irreducible by contradiction and this will  yield that $S_z$ is irreducible.

 Suppose that  $E_{S_z^*}$ is  reducible and $E_{S_z^*}=E^1\oplus E^2$, since $E_{S_z^*}$ is not strictly reducible, there is some $i$ such that either $E^1(\Omega_i)$ or $E^2(\Omega_i)$ equals $\{0\}$. Without loss of generality, we assume $E^1(\Omega_1)=\{0\}$, $E^2(\Omega_1)=E_{S_z^*}(\Omega_1)$, and
$$E_{S_z^*}(\Omega_i)=E^1(\Omega_i)\oplus E^2(\Omega_i),$$
where $E^1(\Omega_i)\perp E^2(\Omega_j)$ for $i\neq j$.
For $\lambda\in \Omega_i$, there is a natural basis of $Ker S_{z-\lambda}^*$ with the form of
\[ K_\lambda(z)K_{a_1(\lambda)}(w),~~K_\lambda(z)K_{a_1(\lambda)}'(w),~~\cdots,~~K_\lambda(z)K_{a_1(\lambda)}^{(k_1)}(w)  \]
\[ K_\lambda(z)K_{a_2(\lambda)}(w),~~K_\lambda(z)K_{a_2(\lambda)}'(w),~~\cdots,~~K_\lambda(z)K_{a_2(\lambda)}^{(k_2)}(w)  \]
\[ \vdots\]
\[ K_\lambda(z)K_{a_r(\lambda)}(w),~~K_\lambda(z)K_{a_r(\lambda)}'(w),~~\cdots,~~K_\lambda(z)K_{a_r(\lambda)}^{(k_r)}(w)  \]
where $\{ a_k(\lambda)\}$  are locally anti-holomorphic functions and $K^{(l)}_\zeta(w)=\frac{l! w^l}{(1-\bar{\zeta}w)^{l+1}}$.
Assume $\{e_1(\lambda),\cdots, e_{m_i}(\lambda)\}$ is a local anti-holomorphic frame of $E^1(\Omega_i)$, then every $e_j(\lambda)$ can be written as a linear combination of the above natural basis with coefficients $f_{kl}^j(\lambda)$, where $f_{kl}^j$ is anti-holomorphic in some open set of $\Omega_i$.  For any $\zeta\in \Omega_1$, there is a neighborhood $\Delta$ of $\zeta$ and some  functions $\rho_n$ holomorphic in $\Delta$ such that $\theta(\zeta,\rho_n(\zeta))=0$ holds on $\Delta$. Since $\theta$ has no univariate inner factors, every function $\rho_n(\zeta)$ is not a constant function.      Note that $K_\zeta(z)K_{\rho_n(\zeta)}(w)\in Ker S_{z-\zeta}^*$  and  $E^1(\Omega_i)\perp E^2(\Omega_1)$, we have
 \[ 0=\langle e_j(\lambda), K_\zeta(z)K_{\rho_n(\zeta)}(w) \rangle=\sum_{k,l} f_{kl}^j(\lambda)K_\lambda(\zeta)K_{a_k(\lambda)}^{(l)}(\rho_n(\zeta))  \]
for any $\lambda\in \Omega_i$ and $\zeta\in \Delta$. Since the range of $\rho_n$ is also an open set, it implies that
\[\sum_{k,l} f_{kl}^j(\lambda) K_{a_k(\lambda)}^{(l)}(w)=0 \]
 holds for any $\lambda\in \Omega_i$ and $w$ in some open set of $\mathbb{D}$.
 But note that for every given $\lambda\in \Omega_i$, the functions $\{~ K_{a_k(\lambda)}^{(l)}(w)~\}$
are linear independent, hence $f_{kl}^j(\lambda)=0$ for any $k,l$ and $\lambda\in \Omega_i$, which implies that $e_j(\lambda)\equiv 0$. Thus all $E^1(\Omega_i)$ are zero vector bundle. This is a contradiction.  Therefore $E_{S_z^*}$ is irreducible. By Theorem \ref{reduce}, we get that $S_z$ is irreducible and this completes the proof.
\end{proof}

\textbf{Remark:}In fact, when $\theta$ has an inner factor of the form $\phi(w)$, then $S_z^*$ is the direct sum of a Cowen-Douglas operator over $\mathbb{D}$ and a generalized Cowen-Douglas operator. For instance, let $\theta(z,w)=b_{w_0}(w)\varphi(z,w)$, where $\varphi$ is a $(n,1)$-type rational inner function without univariate inner factor. Then we have
\[H^2(\mathbb{D}^2)\ominus \theta H^2(\mathbb{D}^2)=\left( H^2(\mathbb{D}^2)\ominus b_{w_0}(w)H^2(\mathbb{D}^2) \right)
\bigoplus b_{w_0}(w)\left(H^2(\mathbb{D}^2)\ominus \varphi H^2(\mathbb{D}^2)\right). \]
Thus $S_z^{\theta}$ is unitarily equivalent to $S_z^{b_{w_0}}\bigoplus S_z^{\varphi}$ (here $S_z^{\theta}$ means the compressed shift operator $S_z$ on the quotient module $\mathcal{K}_\theta$). Suppose the Generalized Cowen-Douglas domain of $(\bigoplus S_z^{\varphi})^*$ is $\Omega$, then  $(S_z^{\theta})^*$ is a Generalized Cowen-Douglas operator over $\mathbb{D}\setminus \partial \Omega$, and $E_{(S_z^{\theta})^*}$ has the following decomposition:
\begin{align*}
  E^1(\Omega)&=span\{~K_{w_0}(w_0)\cdot K_\lambda K_{a(\lambda)}-K_{a(\lambda)}(w_0)\cdot K_\lambda K_{w_0}~\} & E^1(\mathbb{D}\setminus \overline{\Omega})&=\{0\}, \\
  E^2(\Omega)&=span\{~ K_\lambda K_{w_0}~\} &E^2(\mathbb{D}\setminus \overline{\Omega})&=span\{~ K_\lambda K_{w_0}~\}.
\end{align*}

\vspace{1cm}

 By Theorem \ref{strictly-reducible}, if the zeros counter function $Z_\theta$ is $1$ on some connected component  of $\Omega$, then $S_z$ is irreducible. Therefore $S_z^*$ is irreducible on  quotient modules induced by $(n,1)$-type rational inner functions $\theta$ satisfying the condition in Theorem \ref{strictly-reducible}, and this fact was obtained in \cite{ZYL} by a  different method.
The following example is another application of Theorem \ref{strictly-reducible}.
\begin{example}
Let $\theta(z,w)=\prod\limits_{i=1}^n \frac{zw-t_i}{1-t_izw}$  be a rational inner function with $0<t_1<t_2\leq t_3\leq\cdots \leq t_n<t_{n+1}=1$. Then we have
\[\sigma(S_z)\setminus \sigma_e(S_z)=\bigcup_{i=1}^n ~\{ \lambda\in \mathbb{C}: t_i<|\lambda|<t_{i+1} \}. \]
Since $Z_\theta(\lambda)=1$ on $\{ \lambda\in \mathbb{C}: t_1<|\lambda|<t_{2} \}$, by Theorem \ref{4.7}, $S_z$ is irreducible on $\mathcal{K}_{\theta}$.

\end{example}

It is not hard to see that the   proof for Theorem \ref{strictly-reducible} also holds for the polynomial  quotient module $H^2(\mathbb{D}^2)\ominus [p]$ such that $S_z^*$ is a generalized Cowen-Douglas operator. Therefore it is natural to ask the following question.
 \begin{question}
 Suppose $p(z,w)$ is a polynomial   without univariate factors , then when is $S_z^*$ a generalized Cowen-Douglas operator on   $H^2(\mathbb{D}^2)\ominus [p]$?
 \end{question}
This question is equivalent to ask whether the polynomial submodule $[p]$ is determined by its zeros in $\mathbb{D}^2$ (counting  multiplicity)? It seems to be positive for any polynomial but we cannot give a proof at this moment.  If it does, then   $S_z^*$ is a generalized Cowen-Douglas operator on $\Omega$ with index $\alpha_i=Z_p(\lambda), \lambda \in \Omega_i$, where   $\Omega=\sigma(S_z)\setminus \sigma_e(S_z)=\bigcup_i \Omega_i$. In this case,  $S_z$ is reducible if and only if $E_{S_z^*}$ is strictly reducible. We will give a further discussion for certain polynomials in next section.

We end this section with a concrete example that $S_z^*$ (not a Cowen-Douglas operator) is reducible. It is just a modification of Example \ref{counterexample}.
\begin{example}\label{counterexample2}
Take $0<c<1/8$, define
\[ \theta(z,w)= \frac{2z^2w^2-z+c}{2-zw^2+cz^2w^2}. \]
Let
\begin{equation*}
h_1(z,w)=
\begin{cases}
  z-\frac{1-\sqrt{1-8cw^2}}{4w^2} &if\quad w\neq 0\\
  z-c&if\quad w=0
\end{cases}
\end{equation*}
\[ h_2(z,w)=2zw^2-(1+\sqrt{1-8cw^2})/2 ,\]
then
\[\Omega=\pi_1(Z_\theta \cap \mathbb{D}^2)= \Omega_1 \cup \Omega_2=\pi_1(Z_{h_1} \cap \mathbb{D}^2)\cup \pi_1(Z_{h_2} \cap \mathbb{D}^2), \]
and the index of $S_z^*$ is $(2,2)$. Let $\rho(\lambda)=\sqrt{\frac{\lambda-c}{2\lambda^2}}$, then
\begin{equation*}
Ker(S_{z-\lambda}^* )=
\begin{cases}
~~ span~ \{ K_\lambda(z)K_{\rho(\lambda)}(w), K_\lambda(z) K_{-\rho(\lambda)}(w) ~\}
  &if\quad \lambda\in \Omega,\lambda\neq c\\
~~ span~ \{ K_{\lambda}(z), K_\lambda(z)w  \}
  &if\quad \lambda =c
\end{cases}
\end{equation*}
It is not hard to check that
\[ (\lambda, K_\lambda(z)K_{\rho(\lambda)}(w)+ K_\lambda(z) K_{-\rho(\lambda)}(w)) \]
and
\[ \left(\lambda, \frac{1}{2\overline{\rho(\lambda)}}  [K_\lambda(z)K_{\rho(\lambda)}(w)- K_\lambda(z) K_{-\rho(\lambda)}(w))] \right) \]
gives a reducible decomposition of $E_{S_z^*}$, hence by Theorem \ref{strictly-reducible}, $S_z$ is reducible.
\end{example}

\section{ Polynomial quotient modules $[z^m-w^n]^\perp$ and $[(z-w)^n]^\perp$}
 In the section, we will study the compressed shift operators on certain polynomial quotient modules   by using the geometric approach.

\begin{definition}\label{divisor}
  Let $p(z,w)$ be a polynomial.  we say $p$ has divisor property  for $H^2(\mathbb{D}^2)$ if for any $f\in H^2(\mathbb{D}^2)$ such that $f/p$ is holomorphic in $\mathbb{D}^2$, then  $f\in [p]$.
\end{definition}
If  $p$ has no   factors depending only on variable $z$,  $S_z^*$ is a generalized Cowen-Douglas operator on the quotient module $H^2(\mathbb{D}^2)\ominus [p]$ if and only if $p$ has divisor property. In the case of unit disk, every $p(z)$ has the divisor property for $H^2(\mathbb{D})$ since the existence of inner-outer factorization, hence the definition is trivial. For the case of  bidisk, as we have seen, the numerator of a rational inner function has divisor property. For general polynomials, the following theorem in \cite{DGN} gives an affirmative answer.

\begin{theorem}[\cite{DGN}, Theorem 2.1]\label{5.3}
Let $A_{\beta}^2$ be the classical weighted Bergman space over disk if $\beta>-1$ and $A_{-1}^2$ denote the Hardy space. Suppose $\varphi$ is holomorphic in some neighborhood of $\mathbb{D}^n$. Then
\[ [\varphi]_n=\{~~ \varphi f \in A_{\overrightarrow{\beta},n}^2: f~~is~~holomorphic~~in~~\mathbb{D}^n     ~~\}. ~~\]
Here, $\overrightarrow{\beta}=(\beta_i)_{i=1}^n$ with $\beta_i\geq -1$, $A_{\overrightarrow{\beta},n}^2=\bigotimes_{i=1}^n A_{\beta_i}^2$, and $[\varphi]_n$ is the submodule in $A_{\overrightarrow{\beta},n}^2$ generated by $\varphi$.
\end{theorem}

By the above theorem, take $n=2$ and $\overrightarrow{\beta}=(-1,-1)$, we conclude that every polynomial $p(z,w)$ has the divisor property.

It is known that  for a polynomial $q$, if $Z_q\cap \mathbb{D}^2=\emptyset$, then $[q]=H^2(\mathbb{D}^2)$ (cf.\cite{CG}, Proposition 2.2.13). So we always assume every algebraic component of $p$   has a nonempty intersection with $\mathbb{D}^2$. The following theorem is an extension of Theorem \ref{kernel}   and the proof is similar, so we omit it.
\begin{theorem}\label{5.1}
Let $p$ be a polynomial without factors depending only on  $z$ and $M=[p]$. The followings holds.

  \begin{enumerate}
\setlength{\itemsep}{0pt}
\item  For   $\lambda\in \mathbb{T}$,  then
\[ Ker S_{z-\lambda} =Ker S_{z-\lambda}^*=\{0\}.\]

\item For $\lambda\in \mathbb{D}$, we have
\[ Ker S_{z-\lambda}= \{0\}  \]
and
\[Ker S_{z-\lambda}^*=\left\{\frac{g(w)}{1-\overline{\lambda}z}:  g(w)\in H^2\ominus p(\lambda,w)H^2 \right\}.\]

\item The range of $S_{z-\lambda}$ is closed if and only if $p(\lambda, \cdot)$ has no zeros on $\mathbb{T}$. Hence for $\lambda\in \mathbb{D}\setminus \pi_1(Z_p\cap \mathbb{D}\times \mathbb{T})$, $S_{z-\lambda}$ is Fredholm with
\[ ind(S_{z-\lambda}^*)=Z_p(\lambda), \]
where $Z_p(\lambda)$ denotes the number of zeros of $p(\lambda,\cdot)$ in $\mathbb{D}$.
\end{enumerate}
\end{theorem}
Similar to the proof of Theorem \ref{4.7} and Theorem \ref{strictly-reducible}, we have the following corollary.

\begin{theorem}\label{5.5}
Let $p$ be a polynomial without univariate factors, then the compressed shift $S_z^*$ on $[p]^\perp$ is a generalized Cowen-Douglas operator over $\mathbb{D}\setminus \pi_1(Z_p\cap \mathbb{D}\times \mathbb{T})$. Moreover,  $S_z$ is reducible if and only if $E_{S_z^*}$ is strictly reducible.
\end{theorem}

By Theorem \ref{5.5}, for the polynomials $z^m-w^n$ and $(z-w)^n$, the associated compressed shifts are always generalized Cowen-Douglas operators. Moreover, by Theorem \ref{5.1}, it is easy to check that in both two cases, $\sigma(S_z)=\overline{\mathbb{D}}, \sigma_e(S_z)=\mathbb{T}$, therefore the following corollary is immediate.
\begin{corollary}\label{5.4}
 Let $p(z,w)$ be the polynomial $z^m-w^n$ or $(z-w)^n$, then $S_z^*$ on the quotient module $[p]^\perp$ is a Cowen-Douglas operator over $\mathbb{D}$.
\end{corollary}

%

%

In the following, by using the geometric approach, we study the reducing subspaces of $S_z$ on the quotient modules $\mathcal{N}=[z^m-w^n]^\perp$ and $\mathcal{N}=[(z-w)^n]^\perp$. For convenience,   we   view the Hermitian holomorphic vector bundle over $\Omega^*$ as the Hermitian anti-holomorphic vector bundle over $\Omega$.  For a anti-holomorphic cross-section $e(\lambda)$ over $\Omega$, $e^{(k)}(\lambda)$ means $\frac{\partial^k}{\partial \bar{\lambda}^k}e(\lambda)$. We take the negative real axis as the branch line for the function $z^{\frac{1}{n}}$, i.e., if   write $z=re^{i\theta}, -\pi<\theta<\pi$, then $z^{\frac{1}{n}}=r^{\frac{1}{n}} e^{i\theta/n}$. For $j=0,1,\ldots,n-1$,  define
\[e_j(\lambda)=\frac{1}{1-\bar{\lambda}z}\frac{w^j}{1-\bar{\lambda}^m w^n}=\frac{1}{n\bar{\lambda}^{\frac{jm}{n}}} \sum_{k=0}^{n-1}~ \zeta^{jk} K_{\lambda,~ \lambda^{\frac{m}{n}}\zeta^k},\]
where $\zeta=e^{\frac{2\pi i}{n}}$ is the $n$-th primitive roots of unity and $K_{\lambda, \mu}=\frac{1}{1-\bar{\lambda}z} \frac{1}{1-\bar{\mu}w}$ is the reproducing kernel of $H^2(\mathbb{D}^2)$. Then we have the following theorem.

\begin{theorem}\label{zm-wn}
For the quotient module $\mathcal{N}=[z^m-w^n]^\perp$, $S_z^*$ is a Cowen-Douglas operator over $\mathbb{D}$ with index $n$, and   $\{e_j(\lambda)\}_{j=0}^{n-1}$ is an orthogonal anti-holomorphic frame of $E_{S_z^*}$ over $\mathbb{D}$. Hence it gives a minimal reducing subspace decomposition of $S_z:$
\[ \mathcal{N}=\bigoplus_{j=0}^{n-1} \mathcal{H}_j, \]
where $\mathcal{H}_j=\bigvee_{\lambda\in \mathbb{D}} e_j(\lambda)$. 
Moreover, the block diagonalization of the canonical connection on $E_{S_z^*}$ is $M(1, \otimes ~n)$, hence by Theorem \ref{von-alg}, $V^*(S_z)$ is isomorphic to $M_n(\mathbb{C})$.
\end{theorem}
\begin{proof}
For $\lambda\neq 0$, by Theorem \ref{5.1}, $\{~K_{\lambda,~ \lambda^{\frac{m}{n}}\zeta^k}~\}_{k=0}^{n-1}$ is a basis of the fiber $Ker S_{z-\lambda}^*$.  Note that
\begin{equation*}
  \left(
    \begin{array}{c}
      e_0(\lambda) \\
      e_1(\lambda) \\
      \vdots \\
      e_{n-1}(\lambda) \\
    \end{array}
  \right)
  =\left(
     \begin{array}{cccc}
       \frac{1}{n} & 0 & 0 & 0 \\
       0 & \frac{1}{n\bar{\lambda}^{\frac{m}{n}}} & 0 & 0 \\
       0 & 0 & \ddots & 0 \\
       0 & 0 & 0 & \frac{1}{n\bar{\lambda}^{\frac{m(n-1)}{n}}} \\
     \end{array}
   \right)
   \left(
     \begin{array}{cccc}
       1 & 1 & \cdots & 1 \\
       1 & \zeta & \cdots & \zeta^{n-1} \\
        \vdots& \vdots & \vdots & \vdots \\
       1 & \zeta^{n-1} & \cdots & \zeta^{(n-1)^2} \\
     \end{array}
   \right)
   \left(
     \begin{array}{c}
       K_{\lambda, ~\lambda^{\frac{m}{n}}} \\
       K_{\lambda, ~\lambda^{\frac{m}{n}}\zeta} \\
       \vdots \\
       K_{\lambda, ~\lambda^{\frac{m}{n}}\zeta^{n-1}} \\
     \end{array}
   \right),
\end{equation*}
thus $\{e_j(\lambda)\}$ is also a basis of $Ker S_{z-\lambda}^*$.

Since the Fourier coefficients of $e_j(\lambda)$ with respect to  variable $w$ vanishes except for $k\equiv j\mod n$,  we have $$\langle e_i(\lambda),  e_j(\lambda)\rangle=0$$
for $i\neq j$. It is also not hard to check that
\[  \langle e_j(\lambda),  e_j(\lambda)\rangle= \frac{1}{1-|\lambda|^2} \langle \sum_{k=0}^\infty (\bar{\lambda}^m w^n)^k,  \sum_{k=0}^\infty (\bar{\lambda}^m w^n)^k\rangle =\frac{1}{1-|\lambda|^2} \frac{1}{1-|\lambda|^{2m}}.\]
Therefore for the frame $\{e_j(\lambda)\}$, the metric matrix
\[ G=(\langle e_i(\lambda),  e_j(\lambda)\rangle)=\frac{1}{1-|\lambda|^2} \frac{1}{1-|\lambda|^{2m}} I_n.  \]
Hence the connection matrix
\[ \Theta=\bar{\partial} G \cdot G^{-1}=\left(\frac{\lambda}{1-|\lambda|^2}+\frac{m\lambda|\lambda|^{2(m-1)}}{1-|\lambda|^{2m}}\right) \cdot I_n.  \]
Therefore by Theorem \ref{block-diag}, $\mathcal{A}(s)(\lambda)=M(1, \otimes n)$ for all $\lambda \neq 0$. By Theorem \ref{reducing-subspace} and Theorem \ref{von-alg}, the proof is complete.
\end{proof}

\textbf{Remark:} Note that $w^i e_j(\lambda)=w^{j}e_i(\lambda)$, hence $w^i\mathcal{H}_j=w^{j}\mathcal{H}_i$. Write $[z^m-w^n]^\perp=\bigoplus\limits_{j=0}^{n-1} \mathcal{H}_j$, then the operators in $V^*(S_z)$ have the form:
\begin{equation}\label{matrix-form}
  \left(
    \begin{array}{ccccc}
      a_{11}I &a_{12}\bar{w}I & a_{13}\bar{w}^2 I& \cdots & a_{1n}\bar{w}^{n-1}I \\
      a_{21}wI & a_{22}I & a_{23}\bar{w}I & \cdots & a_{2n}\bar{w}^{n-2}I \\
      a_{31}w^2 I & a_{32}wI & a_{33}I& \cdots & a_{3n}\bar{w}^{n-3}I \\
      \vdots & \vdots & \vdots & \ddots & \vdots \\
      a_{n1}w^{n-1}I & a_{n 2}w^{n-2}I & a_{n 3}w^{n-3} I & \cdots & a_{nn}I \\
    \end{array}
  \right),
\end{equation}
where $a_{ij}, i,j=1,\cdots,n$ are arbitrary complex constants.
Thus every rank one projection $(a_i \overline{a_j})$ in $M_n(\mathbb{C})$ corresponds a minimal projection $(a_i w^i \overline{a_j}\bar{w}^j I)$ in $V^*(S_z)$. This fact implies that $S_z$ has infinite minimal reducing subspaces, and all of them are unitarily equivalent.

Moreover, by the spanning property $\mathcal{H}_0=\bigvee_{\lambda\in \mathbb{D}} e_0(\lambda)=\bigvee_{k=0}^\infty e_0^{(k)}(0)$, we can give an orthogonal basis of $\mathcal{H}_0$.  Note that $(\frac{1}{1-\bar{\lambda}z})^{(k)}(0)=k!z^k$, if we write
\[\frac{1}{1-\bar{\lambda}^m w^n}=\sum_{k=0}^\infty (\bar{\lambda}^m w^n)^k,\]
then
\begin{equation*}
(\frac{1}{1-\bar{\lambda}^m w^n})^{(N)}(0)=
\begin{cases}
  ~(mk)!w^{nk}&if\quad N=mk\\
  ~0 &\quad others.
\end{cases}
\end{equation*}
Therefore \begin{align*}
           e_0^{(N)}(0)&=\sum_{k=0}^{k=[\frac{N}{m}]} \dbinom{N}{mk}(mk)!w^{nk} (N-mk)! z^{N-mk}  \\
             & =N! \sum_{k=0}^{k=[\frac{N}{m}]} z^{N-mk} w^{nk},
          \end{align*}
where $[\cdot]$ denotes  the least integer function. Since $e_0^{(N)}(0)$ is homogeneous of degree $N$, $\{e_0^{(N)}(0) \}$  is an orthogonal basis of $\mathcal{H}_0$ and $\{w^j e_0^{(N)}(0)\}$  is an orthogonal basis of $\mathcal{H}_j$. So we obtain the following corollary.
\begin{corollary}


$ \left\{ w^j e_0^{(N)}(0)/\|e_0^{(N)}(0)\|  \right\}_{N=0,j=0}^{\infty, n-1 }$ is an orthonormal basis of the quotient module $[z^m-w^n]^\perp$. Further, $S_z$ is a weighted shift of multiplicity $n$ with weights
\[ \left\{\frac{\sqrt{[\frac{N}{m}]+1}}{\sqrt{[\frac{N+1}{m}]+1}} \right\}_{N=0}^\infty \]
\end{corollary}
\begin{proof}
Note that $\|e_0^{(N)}(0)\|=N! \sqrt{[\frac{N}{m}]+1} $ and
\begin{equation*}
\langle S_z e_0^{(N)}(0), e_0^{(k)}(0) \rangle=
\begin{cases}
 ~(N+1)! N! ([\frac{N}{m}]+1)&if\quad k=N+1\\
 ~0 &\quad others
\end{cases}
\end{equation*}
So
\[S_z \frac{e_0^{(N)}(0)}{\|e_0^{(N)}(0)\|}
=\frac{\sqrt{[\frac{N}{m}]+1}}{\sqrt{[\frac{N+1}{m}]+1}}\frac{e_0^{(N+1)}(0)}{\|e_0^{(N+1)}(0)\|}.\]
\end{proof}


Similar to the proof of Theorem \ref{zm-wn}, one can see that that if $p(z,w)=q(z,w^n)$ for some polynomial $q$ and integer $n\geq 2$, then the kernel space vector bundle of $S_z^*$ is reducible. Therefore we have the following theorem.
\begin{theorem}
Suppose $p(z,w)\in \mathbb{C}[z,w]$ has no factors depending only on  $z$ and $p$ has the divisor property, if $p(z,w)=q(z,w^n)$ for some polynomial $q$ and integer $n\geq 2$, then $S_z$ on $\mathcal{N}=[p]^\perp$ is reducible, and the dimension of the von Neumann  algebra $V^*(S_z)$ is at least  $n^2$.
\end{theorem}

\begin{proof}
Let $\zeta=e^{\frac{2\pi i}{n}}$ be the $n$-th primitive roots of unity. Since $p$ has the form $q(z,w^n)$, by Remark \ref{remark2.8}, for $\lambda\in \Omega=\sigma(S_z)\setminus \sigma_e(S_z)$, the natural basis of $Ker S_{z-\lambda}^*$ has the form
\[\left\{ K_{\lambda}(z)\cdot K_{\zeta^i a_k(\lambda)}^{(l)}(w) , ~~~~i=0,\cdots, n-1 ~~and~~k\geq0, 0\leq l\leq l_k-1\right\}.\]
where $\{a_k(\lambda)\}$ are the zeros in $\mathbb{D}$ of $q(\lambda,\cdot)$ and $l_k$ be the corresponding order of $a_k(\lambda)$.
 Define
\begin{equation*}
  E^i_\lambda=  span \left\{\frac{1}{n~ \overline{a_k(\lambda)}^i} \sum_{j=0}^{n-1} \zeta^{ij} K_{\lambda}(z)\cdot K_{\zeta^j a_k(\lambda)}^{(l)}(w) ,~~~ ~~k\geq0, 0\leq l\leq l_k-1  \right\}
\end{equation*}
and $E^i=\{(\lambda, x)\in \Omega\times \mathcal{N}~|~ x\in E^i_\lambda\}$.
It is sufficient to show that
\begin{equation}\label{decomposition}
   E_{S_z^*}=E^0\oplus E^1\oplus \cdots \oplus E^{n-1}.
\end{equation}

By a direct computation, we have
\begin{align*}
\frac{1}{n~ \overline{a_k(\lambda)}^i} \sum_{j=0}^{n-1} \zeta^{ij}\cdot K_{\zeta^j a_k(\lambda)}(w) &=\frac{1}{n~ \overline{a_k(\lambda)}^i} \sum_{j=0}^{n-1} ~\sum_{s\geq 0}~\zeta^{ij}\cdot \bar{\zeta}^{sj} ~\overline{a_k(\lambda)}^s ~w^s  \\
  & =\frac{1}{n~ \overline{a_k(\lambda)}^i} \sum_{s\geq 0}~\left( \sum_{j=0}^{n-1} \zeta^{(i-s)j} \right) ~\overline{a_k(\lambda)}^s ~w^s \\
  &=\frac{1}{ \overline{a_k(\lambda)}^i} \sum_{t\geq 0} ~\overline{a_k(\lambda)}^{nt+i} ~w^{nt+i}\\
  &=\frac{w^i}{1-\overline{a_k(\lambda)}^n w^n}
\end{align*}
Similarly, we can obtain that
\[ \frac{1}{n~ \overline{a_k(\lambda)}^i} \sum_{j=0}^{n-1} \zeta^{ij}\cdot K^{(l)}_{\zeta^j a_k(\lambda)}(w)= \frac{l! w^{nl+i}}{(1-\overline{a_k(\lambda)}^n w^n)^{l+1}} \]

By an observation, we can find that for the elements in $E^i$, the Fourier coefficients  with respect to variable $w$ vanishes except for $k\equiv i \mod n$, which implies that (\ref{decomposition}) holds. Moreover, let $\mathcal{H}_i$ be the reducing subspace induced by $E^i$, then $\mathcal{H}_i$ and $\mathcal{H}_{i'}$ are unitarily equivalent since $E^i$ is unitarily equivalent to $E^{i'}$. Then it is easy to see that the operators with the form  (\ref{matrix-form}) are in $V^*(S_z)$. Therefore its dimension is at least $n^2$.
\end{proof}

\textbf{Remark:}In fact, the above theorem holds for any polynomial of the form $p(z,B_n(w))$ with $B_n(w)$ is a Blaschke product of order $n\geq 2$.

\vspace{1cm}

Next we will   consider the quotient module  $\mathcal{N}=[(z-w)^n]^\perp$. The main result is the following theorem.
\begin{theorem}\label{5.8}
For the quotient module $\mathcal{N}=[(z-w)^n]^\perp$, $S_z$ is irreducible.
\end{theorem}
To prove the theorem, we need the following result on a basis for generalized eigenspaces.
\begin{lemma}[\cite{CD}, Lemma 1.22] \label{5.10}
If $e_1,\cdots,e_n$ are vector-valued holomorphic functions on $\Omega$ such that $e_1(w),\cdots,e_n(w)$ forms a basis for $Ker(T-w)$ for each $w$ in $\Omega$, then
\begin{enumerate}
  \item $(T-w)e_i^{(k)}(w)=ke_i^{(k-1)}(w)$ for all $k$ and $i=1,2,\cdots, n$.
  \item $e_1(w),\cdots, e_n(w),\cdots, e_1^{(k-1)}(w),\cdots,e_n^{(k-1)}(w)$  form a basis for $Ker (T-w)^k$ for $k\geq 1$ and $w$ in $\Omega$.
\end{enumerate}
\end{lemma}
Now suppose $\Phi$ is an element of $H^\infty_{\mathcal{L}(E_T)}(\Omega)$ for which there exists a bounded operator $X$ in $\{T\}'$ such that $\Gamma_T X=\Phi$ (recall that $\Gamma_T$ is the contractive monomorphism from the  commutant $\{T\}'$ of $T$ into $H^\infty_{\mathcal{L}(E_T)}(\Omega)$, see Proposition \ref{commutant-CD}). If $e_1(w),\cdots, e_n(w)$ is a basis of local holomorphic cross-sections for $E_T$, then by differentiating we obtain
\begin{align*}
  Xe'_i(w) &=(Xe_i(w))'=\Phi(w)e'_i(w)+\Phi'(w)e_i(w) \\
  Xe^{''}_i(w) &=(Xe_i(w))^{''}=\Phi(w)e^{''}_i(w)+2\Phi'(w)e'_i(w)+\Phi^{''}(w)e_i(w)\\
          & ~\vdots ~~~ ~~~~~~~~~~~~~~\vdots \\
 Xe^{(N)}_i(w) &=(Xe_i(w))^{(N)}=\Phi(w)e^{(N)}_i(w)+N\Phi'(w)e^{(N-1)}_i(w)+\cdots+\Phi^{(N)}(w)e_i(w)
\end{align*}
In other words, the block matrix for $X|_{Ker(T-w)^{N+1}}$ relative to the basis $\{e_i^{(j)}(w)\}_{i=1}^n~_{j=0}^N$ is
\begin{equation}\label{block}
  \left(
                                           \begin{array}{ccccc}
                                             \Phi(w) & 0 & 0 & \cdots & 0\\
                                              \Phi'(w)& \Phi(w) & 0 & \cdots & 0\\
                                              \Phi^{''}(w)& 2\Phi'(w) & \Phi(w) & \cdots & 0 \\
                                             \vdots & \vdots & \vdots & & \vdots\\
                                             \Phi^{(N)}(w) & N\Phi^{(N-1)}(w) & \dbinom{N}{2}\Phi^{N-2}(w) & \cdots & \Phi(w) \\
                                           \end{array}
\right).
\end{equation}
Or equivalently,
\[ \left[ \dbinom{i}{j}~\Phi^{(i-j)}(w)) \right], ~~~~here~~\dbinom{i}{j}=0~~if~~i>j.\]

\begin{proof}[Proof of Theorem \ref{5.8}]
When $n=1$, it is  well known that $S_z$ is unitarily equivalent to Bergman shift,  and so it is irreducible. So we assume $n\geq 2$.
 By Theorem \ref{5.1} and Theorem \ref{5.3}, for the quotient module $\mathcal{N}=H^2(\mathbb{D}^2)\ominus[(z-w)^n]$, $S_z^*$ is in $\mathfrak{B}_n(\mathbb{D})$. Let $P$ be a projection in $V^*(S_z)$ and $\Phi=\Gamma_{S_z^*}P$, we will prove our theorem by show that $P$ is either $0$ or $I$. Take
 \[ e_i(\lambda)=K_\lambda(z) K_\lambda^{(i)}(w)=\frac{1}{1-\bar{\lambda}z}\frac{i! w^i}{(1-\bar{\lambda}w)^{i+1}},~~~~~~i=0,1,\cdots,n-1 \]
as an anti-holomorphic frame of $E_{S_z^*}$, then the block matrix for $P|_{Ker(S^*_{z-\lambda})^{N+1}}$ relative to the basis $\{e_i^{(j)}(\lambda)\}_{i=0}^{n-1}~_{j=0}^N$ has the form (\ref{block}).

Note that for $0\leq i\leq n-2$, we have
\begin{align*}
  e'_i(\lambda) & =(K_\lambda(z) K_\lambda^{(i)}(w))'=K'_\lambda(z) K_\lambda^{(i)}(w)+K_\lambda(z) K_\lambda^{(i+1)}(w)\\
   &=\frac{z}{1-\bar{\lambda}z} e_i(\lambda) +e_{i+1}(\lambda).
\end{align*}
Since $ e'_i(\lambda),e_{i+1}(\lambda)$ belong to the quotient module, hence $\frac{z}{1-\bar{\lambda}z} e_i(\lambda)\in \mathcal{N}$ for $0\leq i\leq n-2$, and it implies that
\[ S_z(1-\bar{\lambda}S_z)^{-1}e_i(\lambda)=\frac{z}{1-\bar{\lambda}z} e_i(\lambda), ~~~0\leq i\leq n-2. \]
Suppose $\Phi(\lambda)=[a_{k l}(\lambda) ]_{k,l=0}^{n-1}$ with $a_{kl}(\lambda)$ anti-holomorphic in $\mathbb{D}$. Then we have
\begin{equation*}
  P\left(
    \begin{array}{c}
      e_0(\lambda) \\
      e_1(\lambda) \\
      \vdots \\
      e_{n-1}(\lambda) \\
    \end{array}
  \right)
  =[a_{kl}(\lambda)]\cdot \left(
    \begin{array}{c}
      e_0(\lambda) \\
      e_1(\lambda) \\
      \vdots \\
      e_{n-1}(\lambda) \\
    \end{array}
  \right)
\end{equation*}
and
\begin{equation*}
  P\left(
    \begin{array}{c}
      e'_0(\lambda) \\
      e'_1(\lambda) \\
      \vdots \\
      e'_{n-1}(\lambda) \\
    \end{array}
  \right)
  =[a_{kl}(\lambda)]\cdot \left(
    \begin{array}{c}
      e'_0(\lambda) \\
      e'_1(\lambda) \\
      \vdots \\
      e'_{n-1}(\lambda) \\
    \end{array}
  \right)+
  [a'_{kl}(\lambda)]\cdot \left(
    \begin{array}{c}
      e_0(\lambda) \\
      e_1(\lambda) \\
      \vdots \\
      e_{n-1}(\lambda) \\
    \end{array}
  \right).
\end{equation*}
So for $0\leq i\leq n-2$,
\begin{equation}\label{5.a}
  Pe'_i(\lambda)=\sum_{l=0}^{n-1} a_{il}(\lambda)e'_l(\lambda)+ a'_{il}(\lambda) e_l(\lambda).
\end{equation}
It follows from $PS_z=S_zP$ that
\begin{equation}\label{5.b}
\begin{split}
  Pe'_i(\lambda) & =P\left(S_z(1-\bar{\lambda}S_z)^{-1} e_i(\lambda) +e_{i+1}(\lambda)\right) \\
   & =S_z(1-\bar{\lambda}S_z)^{-1}P e_i(\lambda) +Pe_{i+1}(\lambda)\\
   &=S_z(1-\bar{\lambda}S_z)^{-1}\left( \sum_{l=0}^{n-1} a_{il}(\lambda)e_l(\lambda) \right) +\sum_{l=0}^{n-1} a_{i+1,l}(\lambda)e_l(\lambda).
\end{split}
\end{equation}
Thus by (\ref{5.a}), we have
\begin{equation}\label{5.c}
\begin{split}
  \langle Pe'_i(\lambda), K_{\mu,\mu} \rangle & =\sum_{l=0}^{n-1} a_{il}(\lambda)\frac{l!\mu^{l+1}}{(1-\bar{\lambda}\mu)^{l+3}}+ \sum_{l=0}^{n-1} a_{i+1,l}(\lambda)\frac{l!\mu^{l}}{(1-\bar{\lambda}\mu)^{l+2}} \\
   &=a_{i+1,0}(\lambda)\frac{1}{(1-\bar{\lambda}\mu)^{2}}+\sum_{l=1}^{n-1} \left(\frac{1}{l}a_{i,l-1}(\lambda)+ a_{i+1,l}(\lambda)\right) \frac{l!\mu^{l}}{(1-\bar{\lambda}\mu)^{l+2}}   +a_{i, n-1}(\lambda)\frac{(n-1)!\mu^{n}}{(1-\bar{\lambda}\mu)^{n+2}},
\end{split}
\end{equation}
and by (\ref{5.b}) , we have
\begin{equation}\label{5.d}
\begin{split}
  \langle Pe'_i(\lambda), K_{\mu,\mu} \rangle & =\sum_{l=0}^{n-1} a_{il}(\lambda)\frac{l!\cdot(l+2)\mu^{l+1}}{(1-\bar{\lambda}\mu)^{l+3}}+ a'_{il}(\lambda)\frac{l!\mu^{l}}{(1-\bar{\lambda}\mu)^{l+2}} \\
   &=a'_{i,0}(\lambda)\frac{1}{(1-\bar{\lambda}\mu)^{2}}+\sum_{l=1}^{n-1} \left( \frac{l+1}{l}a_{i,l-1}(\lambda)+a'_{i,l}(\lambda)\right)\frac{l!\mu^{l}}{(1-\bar{\lambda}\mu)^{l+2}}+a_{i,n-1}(\lambda)\frac{n!\mu^n}{(1-\bar{\lambda}\mu)^{n+2}}.
\end{split}
\end{equation}
Since $\left\{ \frac{l!\mu^{l}}{(1-\bar{\lambda}\mu)^{l+2}} \right\}_{l=0}^n$ is linear independent, by (\ref{5.c}) and (\ref{5.d}), we obtain that for~~$i=0,1,\cdots,n-2,~~l=1,\cdots,n-1$,
 \begin{align*}
 &a_{i+1,0}(\lambda)= a'_{i,0}(\lambda),   \\
 &a_{i,l-1}(\lambda)=a_{i+1,l}(\lambda)-a'_{i,l}(\lambda),  \\
 &a_{i,n-1}(\lambda)=0.
  \end{align*}
 For the second formula, let $l=n-1$, we obtain for $i\leq n-3$, $a_{i,n-2}(\lambda)=0$ and $a_{n-2,n-2}(\lambda)=a_{n-1,n-1}(\lambda)$.  Continuing the process, we get
 $a_{00}(\lambda)=\cdots=a_{n-1,n-1}(\lambda)$ and $a_{kl}(\lambda)=0$ for $k<l$.
Thus $[a_{kl}]$ is a lower triangular matrix.  Since $P$ is a projection, $[a_{kl}]$ should be an idempotent, hence the diagonal elements $a_{ii}$ are identically $0$ or $1$. If they are identical zero, an easy computation shows that a lower triangular matrix with diagonal elements are zero is idempotent if and only if it is the zero matrix. If they are identical $1$, then $I-[a_{kl}(\lambda)]$ should   be zero, hence $\Phi(\lambda)$ is $0$ or $I$. That is, the projection contained in $V^*(S_z)$ is either $0$ or $I$, i.e., $S_z$ is irreducible.
\end{proof}

\vspace{0.5cm}
We end this paper with a complete characterization of the reducibility of $S_z^*$ on homogeneous polynomial quotient modules with degree 2. The following technical lemma (Proposition 1, \cite{Englis}) will be used for the proof.
\begin{lemma}\label{6.1}
Let $\Omega$ be a domain in $\mathbb{C}$ and $f(z,w)$ be a function on $\Omega\times \Omega$ which is holomorphic in $z$ and anti-holomorphic in $w$. Then
\[f(z,z)=0\]
for all $z$ in $\Omega$ if and only if $f$ vanishes identically on $\Omega\times \Omega$.
\end{lemma}

\begin{theorem}\label{2-degree}
Let $p=(z-w)(z-\zeta w) (\zeta \neq 0)$ , then $S_z^*$ is reducible on $\mathcal{N}=H^2(\mathbb{D}^2)\ominus[p]$ if and only if $\zeta=-1$.
\end{theorem}
\begin{proof}
If $\zeta=1$, we have known $S_z^*$ is irreducible, so we assume $\zeta\neq 1$. If $|\zeta|<1$,
by the locally property of analytic functions, we have $\bigvee_{|\lambda|<|\zeta|} \{K_{\lambda,\lambda}\}=\bigvee_{|\lambda|<1} \{K_{\lambda,\lambda}\}$, thus $S_z^*$ is a Cowen-Douglas operator over $\{z: |z|<|\zeta|\}$. If $|\zeta|\geq 1$, then $\pi_1(Z_p\cap \mathbb{D}\times \mathbb{T})$ is a empty set, thus $S_z^*$ is a Cowen-Douglas operator over $\mathbb{D}$. So we can  assume $S_z^*$ is a Cowen-Douglas operator on some small disk containing 0. Since $S_z^*$ is reducible, then its kernel space vector bundle  is reducible.
Let $e_1(\lambda), e_2(\lambda)$ be a local orthogonal frame on some neighborhood of $0$, say $\Omega$.  Without loss of generality,   assume
\[ e_1(0)(z,w)=\cos t +\sin t\cdot w,~~~~e_2(0)(z,w)=-\sin t +\cos t\cdot w \]
for some $t\in [0, \pi/2)$.
Note $K_{\lambda,\lambda}, K_{\lambda, \zeta \lambda}$ is a natural basis for $\Omega\setminus \{0\}$, then we can write
\[ e_1(\lambda)=a(\lambda)K_{\lambda,\lambda}+ b(\lambda)  K_{\lambda, \zeta \lambda} \]
\[ e_2(\lambda)=c(\lambda)K_{\lambda,\lambda}+ d(\lambda)  K_{\lambda, \zeta \lambda} \]
where the coefficients $a(\lambda), b(\lambda), c(\lambda), d(\lambda)$ are anti-holomorphic on $\Omega\setminus \{0\}$.

Since $e_1(\lambda)$ is continuous at $\lambda=0$, we have
\begin{equation*}
  \lim_{\bar{\lambda}\rightarrow 0} a(\lambda)K_{\lambda,\lambda}+b(\lambda)K_{\lambda,\zeta\lambda}=\cos t+\sin t\cdot  w,
\end{equation*}
that is,
\begin{equation*}
  \lim_{\bar{\lambda}\rightarrow 0} \frac{1}{1-\bar{\lambda}z}[a(\lambda)\frac{1}{1-\bar{\lambda}w}+b(\lambda)\frac{1}{1-\overline{\zeta\lambda}w}]=\cos t+\sin t\cdot  w.
\end{equation*}
It follows that
\begin{equation*}
  \lim_{\bar{\lambda}\rightarrow 0} \sum_{n\geq 0} [a(\lambda)+\bar{\zeta}^n b(\lambda)]\cdot \bar{\lambda}^n w^n=\cos t+\sin t\cdot w.
\end{equation*}
It implies that  $\lim\limits_{\bar{\lambda}\rightarrow 0} [a(\lambda)+\bar{\zeta}^n b(\lambda)]\bar{\lambda}^n$ exists for all $n\geq 0$ and
\begin{align}
  &  \lim_{\bar{\lambda}\rightarrow 0} ~[a(\lambda)+ b(\lambda)] =\cos t,  \\
 & \lim_{\bar{\lambda}\rightarrow 0} ~[a(\lambda)+ \bar{\zeta}b(\lambda)] \bar{\lambda} =\sin t ,  \\
  &\lim_{\bar{\lambda}\rightarrow 0} ~[a(\lambda)+ \bar{\zeta}^n b(\lambda)] \bar{\lambda}^n =0~~for~~all~~n\geq 2.
\end{align}
Multiply (5.9) by $\bar{\lambda}^{n-1}$ and subtract (5.10), we get
\[ \lim_{\bar{\lambda}\rightarrow 0} ~ (1-\bar{\zeta})\bar{\zeta}^{n-1} b(\lambda) \bar{\lambda}^{n} =0, ~~for~~all~~n\geq 2.\]
This implies that $\bar{\lambda}=0$ is at most a simple pole of $b(\lambda)$. By a similar discussion, it also holds for $a(\lambda), c(\lambda), d(\lambda)$.
Let
\begin{align*}
   a(\lambda)=\sum_{i\geq -1} a_i \bar{\lambda}^i,~~~~ &~~~~ b(\lambda)=\sum_{i\geq -1} b_i \bar{\lambda}^i  \\
   c(\lambda)=\sum_{i\geq -1} c_i \bar{\lambda}^i ,~~~~&~~~~~d(\lambda)=\sum_{i\geq -1} d_i \bar{\lambda}^i
\end{align*}
be the Laurent series of $a(\lambda), b(\lambda), c(\lambda), d(\lambda)$. Then (5.8) and (5.9) implies that
\begin{equation*}
  a_0+b_0=\cos t,~~~a_{-1}+b_{-1}=0,~~~a_{-1}+\bar{\zeta}b_{-1}=\sin t.
\end{equation*}
A similar computation for $c(\lambda), d(\lambda)$ shows that
\begin{equation*}
  c_0+d_0=-\sin t,~~~c_{-1}+d_{-1}=0,~~~c_{-1}+\bar{\zeta}d_{-1}=\cos t.
\end{equation*}
Since $e_1(\lambda) \perp e_2(\lambda)$, we have
\begin{equation}
  \frac{a(\lambda)\overline{c(\lambda)}}{1-|\lambda|^2}+
  \frac{b(\lambda)\overline{d(\lambda)}}{1-|\zeta|^2|\lambda|^2}
  +\frac{a(\lambda)\overline{d(\lambda)}}{1-\zeta|\lambda|^2}
  +\frac{b(\lambda)\overline{c(\lambda)}}{1-\bar{\zeta}|\lambda|^2}=0
\end{equation}
holds for any $\lambda \in \Omega\setminus \{0\}$. By Lemma \ref{6.1}, we have
\begin{equation}\label{5.12}
  \frac{a(\lambda)\overline{c(\mu)}}{1-\bar{\lambda} \mu}+
  \frac{ b(\lambda)\overline{d(\mu)}}{1-|\zeta|^2\bar{\lambda} \mu}
  +\frac{a(\lambda)\overline{d(\mu)}}{1-\zeta\bar{\lambda} \mu}
  +\frac{b(\lambda)\overline{c(\mu)}}{1-\bar{\zeta}\bar{\lambda} \mu}=0,  ~~~\lambda,\mu\in \Omega\setminus \{0\}
\end{equation}
which gives that
\begin{equation}\label{5.13}
  \sum_{n\geq0}~ (a(\lambda)+\bar{\zeta}^n b(\lambda))(\overline{c(\mu)+\bar{\zeta}^nd(\mu)} )\cdot (\bar{\lambda}\mu)^n=0.
\end{equation}
Substituted by the the Laurent series of $a(\lambda), b(\lambda), c(\lambda), d(\lambda)$, we get
\begin{align*}
   &\sum_{n\geq 0} \sum_{i,j \geq -1} (a_i+\bar{\zeta}^n b_i)(\overline{c_j+\bar{\zeta}^nd_j} )\cdot \bar{\lambda}^{n+i} \mu^{n+j} \\
   &=\sum_{k, l\geq -1} \left( \sum_{n\geq 0} ~(a_{k-n}+\bar{\zeta}^n b_{k-n})(\overline{c_{l-n}+\bar{\zeta}^n d_{l-n}} )\right) \bar{\lambda}^{k} \mu^{l}=0.
\end{align*}
Therefore
\[ \sum_{n\geq 0}~ (a_{k-n}+\bar{\zeta}^n b_{k-n})(\overline{c_{l-n}+\bar{\zeta}^n d_{l-n}} )=0 \]
holds for all $k, l\geq -1$. Let $k=0$, we get
\begin{equation}\label{5.14}
  \cos t \cdot(c_l+d_l) +\sin t \cdot(c_{l-1}+\bar{\zeta}d_{l-1})=0~~~~for~~all~~l\geq -1.
\end{equation}
Symmetrically, let $l=0$, we get
\begin{equation}\label{5.15}
  -\sin t \cdot(a_k+b_k) +\cos t \cdot(a_{k-1}+\bar{\zeta}b_{k-1})=0~~~~for~~all~~k\geq -1.
\end{equation}
By (\ref{5.14}) and $(\ref{5.15})$, we have
\[ \cos t \cdot(c(\mu)+d(\mu))+\sin t \cdot\bar{\mu} ~(c(\mu)+\bar{\zeta}d(\mu))=0 \]
\[ -\sin t \cdot(a(\lambda)+b(\lambda))+\cos t \cdot\bar{\lambda} ~(a(\lambda)+\bar{\zeta}b(\lambda))=0 \]
that is,
\begin{equation}\label{5.16}
  a(\lambda)=\frac{\sin t -\bar{\zeta}\bar{\lambda}}{-\sin t +\bar{\lambda} \cos t} b(\lambda),~~~~~~~c(\mu)=\frac{\cos t +\bar{\zeta}\bar{\mu} \sin t}{\cos t +\bar{\mu} \sin t}d(\mu).
\end{equation}
Combine with (\ref{5.13}), we get
\begin{equation*}
  \sum_{n\geq0}~ \left( (1-\bar{\zeta}^n)\sin t +(\bar{\zeta}^n-\bar{\zeta})\bar{\lambda} \cos t \right)\overline{\left((1-\bar{\zeta}^n)\cos t -(\bar{\zeta}^n-\bar{\zeta})\bar{\mu} \sin t\right)} \cdot (\bar{\lambda}\mu)^n=0,
\end{equation*}
Consider the coefficients of $(\bar{\lambda}\mu)^n$, we obtain
\begin{equation*}
  \left( |1-\bar{\zeta}^n|^2-|\bar{\zeta}^{n-1}-\bar{\zeta}|^2 \right)\sin t \cos t=0
\end{equation*}
for all $n\geq 1$.  Recall that $t\in [0, \frac{\pi}{2})$,  we get $\sin t=0$.
By (\ref{5.16}), then $a(\lambda)=-\bar{\zeta}b(\lambda)$, $c(\mu)=-d(\mu)$. Combine with (\ref{5.12}), we have
\begin{equation*}
  \frac{\bar{\zeta}}{1-|\lambda|^2}+
  \frac{1}{1-|\zeta|^2|\lambda|^2}
  -\frac{\bar{\zeta}}{1-\zeta|\lambda|^2}
  -\frac{1}{1-\bar{\zeta}|\lambda|^2}=0,~~~\lambda\in \Omega\setminus\{0\},
\end{equation*}
which gives that $\zeta=-1$.

Conversely, by Theorem \ref{zm-wn}, for quotient module $[z^2-w^2]^\perp$, $S_z^*$ is reducible. The proof is complete.
\end{proof}

\begin{corollary}
Suppose $p=(z-\alpha w)(z-\beta w),  \alpha,\beta \neq 0$ is a homogeneous polynomial. Then on the quotient module $[p]^\perp$, $S_z$ is reducible if and only if $\alpha+\beta=0$.
\end{corollary}
\begin{proof}
If $\alpha\neq \beta$, note that there exist anti-holomorphic functions $a(\lambda), b(\lambda), c(\lambda), d(\lambda)$ such that
\begin{equation*}
   \langle a(\lambda)K_{\lambda, \alpha \lambda}+b(\lambda) K_{\lambda, \beta \lambda} , c(\lambda)K_{\lambda, \alpha \lambda} +d(\lambda)K_{\lambda, \beta \lambda}  \rangle =0
\end{equation*}
 is equivalent to say that there exist anti-holomorphic functions $a_1(\lambda), b_1(\lambda), c_1(\lambda), d_1(\lambda)$ such that
\begin{equation*}
   \langle a_1(\lambda)K_{\lambda,  \lambda}+b_1(\lambda) K_{\lambda, \frac{\beta}{\alpha} \lambda} , c_1(\lambda)K_{\lambda,  \lambda} +d_1(\lambda)K_{\lambda, \frac{\beta}{\alpha} \lambda}  \rangle =0.
\end{equation*}
This implies that the kernel space vector bundle of $S_z^*$ on $[(z-\alpha w)(z-\beta w)]^\perp$ is reducible if and only if the kernel space vector bundle of $S_z^*$ on $[(z- w)(z-\frac{\beta}{\alpha} w)]^\perp$ is reducible. Hence by Theorem \ref{2-degree}, $S_z$ is reducible if and only if $\alpha+\beta=0$.

If $\alpha=\beta$, by a similar discussion, the kernel space vector bundle of $S_z^*$ on $[(z-\alpha w)^2]^\perp$ is reducible if and only if the kernel space vector bundle of $S_z^*$ on $[(z- w)^2]^\perp$ is reducible. By Theorem \ref{5.8}, in this case $S_z$ is irreducible. The proof is complete.
\end{proof}

\end{document}